\documentclass{amsart}

\usepackage{amsmath}
\usepackage{amssymb}
\usepackage{mathrsfs}
\usepackage{graphicx}
\usepackage{placeins}
\usepackage{microtype}
\microtypesetup{expansion=false}
\usepackage[hidelinks]{hyperref}
\numberwithin{equation}{section}
\allowdisplaybreaks
\newtheorem{theorem}{Theorem}[section]
\newtheorem{lemma}[theorem]{Lemma}
\newtheorem{proposition}[theorem]{Proposition}
\newtheorem{corollary}[theorem]{Corollary}
\theoremstyle{definition}
\newtheorem{assumption}[theorem]{Assumption}

\theoremstyle{remark}
\newtheorem{remark}[theorem]{Remark}
\newcommand{\R}{\mathbb R}
\newcommand{\Th}{\mathcal T_h}
\newcommand{\Eh}{\mathcal E_h}
\newcommand{\dual}[2]{\left\langle #1,#2\right\rangle}
\newcommand{\DivDiv}{\operatorname{div}\operatorname{div}}
\newcommand{\symcurl}{\operatorname{symcurl}}
\newcommand{\curl}{\operatorname{curl}}
\newcommand{\dist}{\operatorname{dist}}
\newcommand{\supp}{\operatorname{supp}}
\begin{document}
\title[Localized maximum-norm error estimates for HHJ]
{Localized maximum-norm error estimates for the Hellan--Herrmann--Johnson method}
\author{Yuwen Li}
\address{School of Mathematical Sciences, Zhejiang University,
Hangzhou, Zhejiang 310058, China}
\email{liyuwen@zju.edu.cn}
\author{Zhuoran Teng}
\address{School of Mathematical Sciences, Zhejiang University,
Hangzhou, Zhejiang 310058, China}
\email{tengzhuoran@zju.edu.cn}
\thanks{This work was supported by the National Natural Science
Foundation of China under grant 12471346.}

\begin{abstract}
We derive localized maximum-norm bounds for bending moments computed
by the Hellan--Herrmann--Johnson method for the clamped Kirchhoff plate
problem.  The main difficulty is to localize the discrete stress without
leaving the HHJ space or violating its kernel constraint.  Using
symmetric-curl potentials, we construct a kernel-preserving localization
and connect local Green stresses with a global discrete Green stress.
This argument separates the local interpolation error from a weaker
global pollution term.

Under explicit regularity assumptions on the auxiliary problems, the
resulting estimates give optimal pointwise convergence for bending-moment
values and elementwise first derivatives.  No logarithmic loss occurs for
positive polynomial degrees, whereas the lowest-order value estimate
retains a logarithmic factor.  Under additional reflection symmetry,
symmetric recovery improves bending-moment values for even polynomial
degrees and first derivatives for odd degrees.  The proved gains are one
third and one half of an order, respectively.  Numerical experiments
confirm this parity dependence and exhibit gains close to one full order,
exceeding those established theoretically.
\end{abstract}

\maketitle

\section{Introduction}
Pointwise error estimates describe the local accuracy of finite
element approximations and complement energy- and $L^2$-norm
estimates. Interior estimates are particularly useful when the
solution is smooth near a target region but has singularities
elsewhere: they separate local approximation from weaker global
pollution \cite{NitscheSchatz1974,SchatzWahlbin1995,Demlow}.
For the Hellan--Herrmann--Johnson (HHJ) method, establishing such
estimates requires a localization procedure compatible with
the discrete stress space and its kernel constraint.

The HHJ method originated in two independent plate-bending formulations
published in 1967 \cite{Hellan1967,Herrmann1967}.  It avoids
$C^1$-conforming displacement elements by introducing the bending moment
as a second unknown and imposing only the interelement continuity of its
normal--normal component.  Convergence of the mixed approximation for
plate bending moments was established subsequently
\cite{Johnson1973}.  These works established the stress--deflection
formulation that is now called the HHJ method.

The HHJ scheme was later placed in a stability theory for mixed
discretizations of fourth-order elliptic problems.  The resulting
mesh-dependent norms yield stability and quasi-optimal a priori
estimates from approximation properties of the moment and deflection
spaces \cite{BabuskaOsbornPitkaranta}.  Later work related the HHJ method
to a modified Morley discretization and developed postprocessing and
further error estimates \cite{ArnoldBrezzi1985}, obtained additional
error and supercloseness estimates for postprocessed quantities
\cite{Comodi}, and established stability and convergence for the first
Herrmann, or Herrmann--Johnson, scheme under general polygonal boundary
configurations \cite{BlumRannacher1990}.

More recent developments include the discrete Helmholtz and exact
sequence structure of the HHJ method
\cite{ChenHuHuang2018}, superconvergence based on this structure
\cite{HuMaMa2021}, curved-element analysis
\cite{ArnoldWalker}, and recovery-based a posteriori estimates
\cite{LiRecovery2021}.  These analyses provide global energy,
$L^2$, mesh-dependent, superconvergence, or a posteriori information,
but do not yield localized maximum-norm estimates for the HHJ bending
moment or its elementwise derivatives.

For conforming Galerkin methods for second-order elliptic problems,
Nitsche and Schatz established interior estimates that depend
on local approximation and a weaker global error
\cite{NitscheSchatz1974}.  Schatz and Wahlbin sharpened this principle to
interior $L^\infty$ and $W^{1,\infty}$ bounds of the form ``local best
approximation plus weak outside influence plus local residual''
\cite{SchatzWahlbin1995}.  On meshes locally symmetric about the
observation point, cancellation of the leading local error can further
produce pointwise superconvergence
\cite{SchatzSloanWahlbin}.

Classical studies of plate discretizations established pointwise
convergence and subsequently derived maximum-norm estimates for
nonconforming and mixed approximations
\cite{Rannacher1976,Rannacher1979}.  In the latter reference, the mixed
maximum-norm analysis concerns the Herrmann--Miyoshi scheme, whose
moment space is componentwise conforming.  Interior estimates were also
obtained for a different mixed approximation of fourth-order problems
\cite{Scholz1979}.  These results established Green-function
and interior-estimate techniques for plate problems, but they do not
address the normal--normal conforming HHJ moment space or its discrete
kernel structure.

Pointwise estimates for mixed methods were first developed mainly for
second-order elliptic equations and $H(\operatorname{div})$-conforming
fluxes.  Early work derived $L^q$ error estimates throughout
$1\le q\le\infty$ \cite{Duran} and sharp maximum-norm estimates for both
scalar and flux variables, together with superconvergent elementwise
postprocessing, for the Raviart--Thomas and Brezzi--Douglas--Marini
families \cite{RaviartThomas,BDM,GastaldiNochetto}.  Later weighted
localized pointwise estimates separated local regularity from weak
global dependence and identified the exceptional nonlocal behavior of
the lowest-order BDM method \cite{Demlow}.

The second-order mixed analyses rely on the structure of
$H(\operatorname{div})$-conforming flux spaces and commuting
divergence projections. The HHJ stress space instead enforces
normal--normal continuity and is linked to a symmetric-curl
potential space through a discrete exact sequence. Multiplication
by a smooth cutoff generally leaves the discrete space, while
zero extension across an artificial boundary can violate
normal--normal continuity. Projecting the localized tensor back
into the HHJ space does not, in general, preserve the discrete
kernel. The main difficulty is therefore to couple local
pointwise Green stresses to the global HHJ kernel equation while
retaining this constraint.

To the best of our knowledge, no previous analysis provides
Schatz--Wahlbin-type localized maximum-norm estimates for the HHJ
bending moment that separate local approximation from a weaker
global pollution term. We establish such estimates for interior
bending-moment values and elementwise first derivatives under
explicit regularity assumptions on continuous auxiliary problems.
Using the symmetric-curl representation of the discrete kernel
\cite[Lemma~2.6]{ChenHuHuang2018}, we construct a kernel-preserving
localization and couple local Green stresses to a global discrete
Green stress. This construction separates the local
HHJ interpolation error from a global $H^{-1}$ pollution term
without introducing an auxiliary HHJ problem on an artificial
boundary. The required regularity estimates are stated for
the continuous auxiliary potential and adjoint problems.

Under additional reflection-symmetry assumptions, we obtain
parity-dependent recovery estimates by pairing contributions
from reflected elements. The recovered values depend on the
even part of the interpolation error, whereas the recovered
first derivatives depend on its odd part. This distinction
yields improved value estimates for even stress degrees and
improved first-derivative estimates for odd stress degrees. Numerical
experiments further examine this parity mechanism under reflection
and its dependence on the local symmetry of the mesh.

These conclusions are proved for a constant compliance tensor on
shape-regular, quasi-uniform triangulations and for interior target
regions satisfying $\mathcal O_{4d}\Subset\Omega$, under the regularity
conditions specified below.  The interior restriction permits the use
of local Green stresses on balls compactly contained in $\Omega$; extending
the argument to the boundary would instead require clamped boundary
Green stresses with estimates uniform in the observation point.

\subsection{Main results}

We summarize the main results of the paper.  We write $A\lesssim B$ if $A\le CB$, where the constant $C$ is
independent of the mesh size, the localization scale, and the
observation point. For a piecewise smooth tensor field $q$, we write
\[
\|q\|_{L^\infty(D;\mathcal T_h)}
:=\max_{\substack{T\in\mathcal T_h\\T\cap D\ne\emptyset}}
\|q\|_{L^\infty(T\cap D)}.
\]
Let $\mathcal O\Subset\Omega$ and choose $d\ge\kappa_0h$ such that
$\mathcal O_{4d}\Subset\Omega$.  For the lowest-order HHJ method,
assuming the potential $H^2$ shift
\eqref{eq:HHJ-potential-H2-shift}, we obtain
\begin{equation*}
 \|\sigma-\sigma_h\|_{L^\infty(\mathcal O;\mathcal T_h)}
 \lesssim
 h\bigl(1+\log(d/h)\bigr)
 |\sigma|_{W^{1,\infty}(\mathcal O_{4d})}
 +
 hd^{-1}|\sigma|_{H^1(\Omega)}.
\end{equation*}

For $p\ge1$, assume in addition the potential $H^3$ shift
\eqref{eq:HHJ-potential-H3-shift} and the global HHJ adjoint regularity
\eqref{eq:HHJ-global-adjoint-regularity}.  Then
\begin{align*}
 \|\sigma-\sigma_h\|_{L^\infty(\mathcal O;\mathcal T_h)}
 &\lesssim
 h^{p+1}
 |\sigma|_{W^{p+1,\infty}(\mathcal O_{4d})}
 +
 h^{p+2}d^{-2}
 |\sigma|_{H^{p+1}(\Omega)},
 \\
 \|\nabla\sigma-\nabla_h\sigma_h\|_
 {L^\infty(\mathcal O;\mathcal T_h)}
 &\lesssim
 h^p
 |\sigma|_{W^{p+1,\infty}(\mathcal O_{4d})}
 +
 h^{p+2}d^{-3}
 |\sigma|_{H^{p+1}(\Omega)}.
\end{align*}
In particular, if $d$ is fixed independently of $h$, then we have
\begin{equation*}
 \|\sigma-\sigma_h\|_{L^\infty(\mathcal O;\mathcal T_h)}
 =
 \begin{cases}
 O\!\left(h(1+|\log h|)\right),&p=0,\\
 O(h^{p+1}),&p\ge1,
 \end{cases}
\end{equation*}
The corresponding gradient rate is
\begin{equation*}
 \|\nabla\sigma-\nabla_h\sigma_h\|_
 {L^\infty(\mathcal O;\mathcal T_h)}
 =
 O(h^p),
 \qquad p\ge1.
\end{equation*}

On locally reflection-symmetric meshes, we also obtain
parity-dependent superconvergence.  Under
Assumptions~\ref{ass:HHJ-local-reflection-symmetry}
and~\ref{ass:HHJ-Galerkin-reflection}, the optimized recovered rates are
\begin{align*}
 h^{p+4/3}
 &\qquad
 (p\ge2,\ p\ \text{even; recovered values}),
 \\
 h^{p+1/2}
 &\qquad
 (p\ge1,\ p\ \text{odd; recovered first derivatives}).
\end{align*}
The stronger orders $h^{p+2}$ and $h^{p+1}$ observed numerically on
uniform symmetric meshes are not proved here.

Theorem~\ref{thm:HHJ-local-maximum-norm} establishes the localized
estimates from which the remaining results follow.
Corollary~\ref{cor:HHJ-interior-apriori-rates} inserts the interpolation
bounds to obtain the preceding a priori rates, whereas
Corollary~\ref{cor:HHJ-optimized-superconvergence} combines the localized
estimates with reflection symmetry to obtain the optimized recovery
rates.

The paper is organized as follows.  Section~2 develops the HHJ
framework, the local Green stresses, and the kernel-preserving
localization operator.  Section~3 constructs the global discrete Green stress and
proves the localized maximum-norm estimates.  Section~4 proves the global
negative-norm estimate and the localized a priori rates, and then derives
the parity-dependent recovery estimates.  Section~5 reports the
numerical experiments.

\section{Localized Green's functions}
\label{sec:HHJ-framework-local-Green}
\subsection{The HHJ method}
Let $\Omega\subset\R^2$ be a bounded convex polygon, and let
$\mathcal T_h$ be a conforming, shape-regular triangulation of $\Omega$.
Let $\mathbb S$ be the space of real symmetric $2\times2$ matrices,
equipped with the Frobenius product and norm.  For a piecewise smooth
field $q$, set $\nabla_hq|_T:=\nabla(q|_T)$,
$\partial_\nu q:=\nu\cdot\nabla q$ for $\nu\in\R^2$, and
$(q)_T:=|T|^{-1}\int_Tq\,dx$.
For every $T\in\mathcal T_h$, set $h_T:=\operatorname{diam}(T)$ and
$h:=\max_{T\in\mathcal T_h}h_T$, and let $\mathcal E_h$ be the set of
all mesh edges.  On an interior edge
$e=T^+\cap T^-$, fix a unit normal $n_e$ directed from $T^+$ to $T^-$ and
define $ [\partial_n v]_e
 :=
 (\nabla v|_{T^+}-\nabla v|_{T^-})\cdot n_e$.
On a boundary edge $e\subset\partial T\cap\partial\Omega$, let $n_e$ be
the outward unit normal and set
$[\partial_n v]_e:=\nabla v|_T\cdot n_e$.  In both cases,
$\tau_{nn}:=n_e^T\tau n_e$.

Consider the clamped Kirchhoff plate problem
\begin{equation}
 \sigma=\mathcal C\nabla^2u,
 \qquad
 \DivDiv\sigma=f\quad\text{in }\Omega,
 \qquad
 u=\partial_n u=0\quad\text{on }\partial\Omega.
 \label{eq:plate-strong}
\end{equation}
Let $\mathcal K=\mathcal C^{-1}$.  For piecewise smooth tensors and scalar functions
$v$ that are continuous on $\Omega$ and belong to $H^2(T)$ on every
$T\in\mathcal T_h$, define
\begin{equation*}
 b_h(\tau,v)
 :=-\sum_{T\in\Th}(\tau,\nabla^2v)_T
   +\sum_{e\in\Eh}\langle\tau_{nn},[\partial_n v]\rangle_e.
\end{equation*}
With the displayed sign convention, $b_h$ is the broken realization of
the distributional pairing $-\dual{\DivDiv\tau}{v}$.

The standard degree-$p$ HHJ pair
\cite{BabuskaOsbornPitkaranta,Comodi,BoffiBrezziFortin} is
\begin{align*}
 \Sigma_h
 &:=
 \{\tau_h:\tau_h|_T\in\mathcal P_p(T;\mathbb S),\
          \tau_{h,nn}\text{ is single-valued on every edge}\},
 \\
 V_h
 &:=
 \{v_h\in H_0^1(\Omega):v_h|_T\in\mathcal P_{p+1}(T)\}.
\end{align*}

Let $\Pi_h$ be the canonical HHJ stress interpolant and let $I_h$
denote the HHJ-compatible scalar interpolation operator into $V_h$
\cite{BabuskaOsbornPitkaranta,Comodi,ArnoldWalker}.
Their commuting orthogonalities are
\begin{equation}
\begin{aligned}
b_h(\tau-\Pi_h\tau,v_h)&=0,
&&v_h\in V_h,\\
b_h(\tau_h,v-I_hv)&=0,
&&\tau_h\in\Sigma_h.
\end{aligned}
\label{eq:HHJ-Fortin}
\end{equation}

For $T\in\mathcal T_h$, let $\omega_T$ denote a fixed-layer element
patch containing $T$, with a number of layers independent of $h$.
The local approximation estimates for the canonical HHJ interpolant
\cite{BabuskaOsbornPitkaranta,Comodi} state that, for
$q\in\{1,2,\infty\}$, $j\in\{0,1\}$, and $j\le s\le p+1$,
\[
\|\tau-\Pi_h\tau\|_{W^{j,q}(T)}
\lesssim
h_T^{s-j}|\tau|_{W^{s,q}(\omega_T)},
\]
while the HHJ-compatible scalar interpolant $I_h$ satisfies, for
$q\in\{1,2\}$ and $2\le s\le p+2$,
\[
\|v-I_hv\|_{L^q(T)}
\lesssim
h_T^s |v|_{W^{s,q}(\omega_T)}.
\]

Define the continuous HHJ stress space and deflection space by
\[
 \Sigma:=\left\{\tau\in L^2(\Omega;\mathbb S):
 \DivDiv\tau\in H^{-1}(\Omega)\right\},
 \qquad V:=H_0^1(\Omega),
\]
with the graph norm on $\Sigma$.  The continuous HHJ bilinear form is
\[
 b(\tau,v):=\langle-\DivDiv\tau,v\rangle_{H^{-1},H_0^1}.
\]
For every open set $D\subset\Omega$, set
$a_D(\tau,\varphi):=(\mathcal K\tau,\varphi)_D$ and $a:=a_\Omega$.
The tensor $\mathcal K$ is bounded, symmetric, and uniformly positive
definite.

These spaces and bilinear forms give the continuous HHJ formulation of
\eqref{eq:plate-strong}:
\begin{equation*}
 \begin{aligned}
  a(\sigma,\tau)+b(\tau,u)&=0,&&\tau\in\Sigma,\\
  b(\sigma,v)&=-(f,v),&&v\in V.
 \end{aligned}
\end{equation*}

Replacing $\Sigma\times V$ and $b$ by their discrete counterparts yields
the discrete HHJ problem: find
$(\sigma_h,u_h)\in\Sigma_h\times V_h$ such that
\begin{equation*}
 \begin{aligned}
  a(\sigma_h,\tau_h)+b_h(\tau_h,u_h)&=0,&&\tau_h\in\Sigma_h,\\
  b_h(\sigma_h,v_h)&=-(f,v_h),&&v_h\in V_h.
 \end{aligned}
\end{equation*}

\subsection{Discrete kernel and the equation for the discrete stress error}
For a vector field $r=(r_1,r_2)$ we use the two-dimensional
symmetric curl
\[
 \symcurl r
 :=
 \frac12\bigl(\curl r+(\curl r)^T\bigr),
 \qquad
 \curl r:=
 \begin{pmatrix}
  \partial_2r_1&-\partial_1r_1\\
  \partial_2r_2&-\partial_1r_2
 \end{pmatrix}.
\]
The identity $\DivDiv(\symcurl r)=0$ holds in the sense of distributions
and motivates the following potential representation of the discrete
kernel.  Define the global conforming potential space and its affine
kernel by
\[
\begin{aligned}
 W_h(\Omega)&:=
 \left\{w_h\in H^1(\Omega;\mathbb R^2):
 w_h|_T\in\mathcal P_{p+1}(T;\mathbb R^2)\right\},\\
 W_{h,0}&:=W_h(\Omega)\cap H_0^1(\Omega;\mathbb R^2),\\
\mathcal R(D)&:=
 \{r\in H^1(D;\mathbb R^2):\symcurl r=0\}.
\end{aligned}
\]
We also write
\[
 \operatorname{RM}(D)
 :=\{q\in H^1(D;\mathbb R^2):\operatorname{sym}\nabla q=0\}
\]
for the rigid motions on $D$.
On every connected patch $D$, the equation $\symcurl r=0$ identifies
this kernel with the affine space
\begin{equation*}
 \mathcal R(D)=\{r(x)=\alpha x+b:\ \alpha\in\mathbb R,\ b\in\mathbb R^2\}.
\end{equation*}

\begin{proposition}[Global HHJ kernel and potential representation]
Define
\[
Z_h
:=
\{\tau_h\in\Sigma_h:
  b_h(\tau_h,v_h)=0
  \quad\forall v_h\in V_h\}.
\]
We have $Z_h=\symcurl W_h(\Omega)$.  Moreover, for every
$\tau_h\in Z_h$, there exists a unique
$R_h\tau_h\in W_h(\Omega)$ satisfying
\begin{equation*}
\symcurl R_h\tau_h=\tau_h,
\qquad
(R_h\tau_h,r)_\Omega=0
\quad\forall r\in\mathcal R(\Omega).
\end{equation*}
Thus $R_h:Z_h\longrightarrow W_h(\Omega)$ is a normalized right inverse of $\symcurl$.
\end{proposition}

\begin{proof}
The discrete HHJ exact sequence on the contractible domain $\Omega$
gives $Z_h=\symcurl W_h(\Omega)$
\cite[Lemma~2.6]{ChenHuHuang2018}.  The kernel of $\symcurl$ on
$W_h(\Omega)$ is $\mathcal R(\Omega)$.
Hence the $L^2(\Omega)$-orthogonality conditions against
$\mathcal R(\Omega)$ determine a unique potential representative.
\end{proof}

We shall also use a boundary-preserving Scott--Zhang quasi-interpolation
operator $J_h:H^1(\Omega;\mathbb R^2)\longrightarrow W_h(\Omega)$ as in \cite[Chapter~4]{BrennerScott2008}.
It is a projection on $W_h(\Omega)$, is local on uniformly bounded
element patches, and satisfies
\begin{equation*}
J_h\bigl(H_0^1(\Omega;\mathbb R^2)\bigr)
\subset W_{h,0},
\qquad
\supp(J_h\phi)\subset(\supp\phi)^+.
\end{equation*}
In particular, we have
\begin{equation*}
J_h w_h=w_h
\quad\forall w_h\in W_h(\Omega),
\qquad
J_h r=r
\quad\forall r\in\mathcal R(\Omega).
\end{equation*}
For $1\le s\le p+2$, its local approximation property is
\begin{equation*}
|v-J_hv|_{H^1(T)}
\lesssim
h_T^{s-1}|v|_{H^s(\omega_T)},
\qquad
v\in H^s(\omega_T;\mathbb R^2).
\end{equation*}
For $p\ge1$, affine reproduction and local $H^1$ stability also give
the following scaled local $H^2$ bound:
\begin{equation}
 |J_h v|_{H^2(T)}
 \lesssim
 |v|_{H^2(\omega_T)},
 \qquad
 v\in H^2(\omega_T;\mathbb R^2).
 \label{eq:HHJ-Scott-Zhang-H2-stability}
\end{equation}
Indeed, let $r_T\in\mathcal P_1(\omega_T;\mathbb R^2)$ be a local affine
approximant to $v$.  Since $J_h r_T=r_T$, the inverse estimate, local
$H^1$ stability, and the Bramble--Hilbert estimate give
\begin{align*}
 |J_h v|_{H^2(T)}
 &=
 |J_h(v-r_T)|_{H^2(T)}
 \lesssim
 h_T^{-1}|J_h(v-r_T)|_{H^1(T)}
 \\
 &\lesssim
 h_T^{-1}|v-r_T|_{H^1(\omega_T)}
 \lesssim
 |v|_{H^2(\omega_T)}.
\end{align*}
For a cutoff $\eta$ satisfying
$\|D^j\eta\|_{L^\infty}\lesssim d^{-j}$,
the standard scaled superapproximation estimate gives
\begin{equation*}
\left\|
\symcurl\bigl(\eta w_h-J_h(\eta w_h)\bigr)
\right\|_{L^2(T)}
\lesssim
\frac hd
\left(
\|\nabla w_h\|_{L^2(\omega_T)}
+d^{-1}\|w_h\|_{L^2(\omega_T)}
\right).
\end{equation*}

For later use, we define the potential bilinear form for
$\phi,\psi\in H^1(D;\mathbb R^2)$ by
\begin{equation*}
 \mathcal A_D(\phi,\psi)
 :=
 a_D(\symcurl\phi,\symcurl\psi).
\end{equation*}
By uniform ellipticity of $\mathcal K$, we have
\begin{equation*}
 c\|\symcurl\phi\|_{L^2(D)}^2
 \le
 \mathcal A_D(\phi,\phi)
 \le
 C\|\symcurl\phi\|_{L^2(D)}^2.
\end{equation*}
Moreover, $\mathcal A_D(\phi,\phi)=0$ if and only if
$\phi\in\mathcal R(D)$.  Thus $\mathcal A_D$ is a bounded and coercive
bilinear form on the quotient space
$H^1(D;\mathbb R^2)/\mathcal R(D)$, equipped with the norm induced by
$\|\symcurl(\cdot)\|_{L^2(D)}$.  This potential structure also
identifies the component of the stress error controlled through the
discrete kernel.

We split the stress error as
\begin{equation*}
 \sigma-\sigma_h
 =
 \rho_\sigma+\xi_h,
 \qquad
 \rho_\sigma:=\sigma-\Pi_h\sigma,
 \qquad
 \xi_h:=\Pi_h\sigma-\sigma_h.
\end{equation*}
Here $\rho_\sigma$ is the interpolation error, whereas $\xi_h$ is the
discrete error relative to the commuting projection.

\begin{lemma}[Kernel equation for the discrete stress error]
The errors $\rho_\sigma$ and $\xi_h$ satisfy
\begin{align*}
a(\xi_h,\tau_h)
+
b_h(\tau_h,I_hu-u_h)
&=
-a(\rho_\sigma,\tau_h),
&&
\tau_h\in\Sigma_h,
\\
b_h(\xi_h,v_h)
&=0,
&&
v_h\in V_h.
\end{align*}
Thus, $\xi_h\in Z_h$, and
\begin{equation}
a(\xi_h,\tau_h)
=
-a(\rho_\sigma,\tau_h),
\qquad
\tau_h\in Z_h.
\label{eq:kernel-error}
\end{equation}
\end{lemma}

\begin{proof}
The discrete--continuous HHJ error equations are
\begin{align*}
a(\sigma-\sigma_h,\tau_h)
+b_h(\tau_h,u-u_h)
&=0,
&\qquad \tau_h&\in\Sigma_h,\\
b_h(\sigma-\sigma_h,v_h)
&=0,
&\qquad v_h&\in V_h.
\end{align*}

Substituting $\sigma-\sigma_h=\rho_\sigma+\xi_h$ and using the HHJ
commuting orthogonalities
\eqref{eq:HHJ-Fortin},
\[
b_h(\rho_\sigma,v_h)=0,
\qquad
b_h(\tau_h,u-I_hu)=0,
\]
gives the two asserted error equations.  The second equation places
$\xi_h$ in $Z_h$, and testing the first equation with
$\tau_h\in Z_h$ then yields \eqref{eq:kernel-error}.
\end{proof}

For the remainder of this section, the triangulation is assumed to be
shape-regular and quasi-uniform.  We now develop the local constructions
needed for the interior maximum-norm estimate.

\subsection{Regularized point functionals and local Green stresses}

For any set $D\subset\Omega$ and $t>0$, set
$D_t:=\{x\in\Omega:\dist(x,D)<t\}$.  Fix $d\ge\kappa_0 h$ and assume
$\mathcal O_{4d}\Subset\Omega$.  For each $x_0\in\mathcal O$, define
$U=U_{x_0}:=B_{5d/2}(x_0)$.  Then
$B_{2d}(x_0)\subset U\subset\mathcal O_{3d}$.

For any $D\subset\Omega$, let $D^+$ and $D^{++}$ denote fixed-layer
element enlargements, formed from the elements meeting $D$ and their
neighbors, with layer counts independent of $h$ and $d$.
These connected, uniformly Lipschitz enlargements provide the separated
domains required by the local estimates.

If $x_0$ lies on the mesh skeleton, we fix an adjacent element $T_0$
with $x_0\in\overline{T_0}$ and interpret all point values and
derivatives at $x_0$ by the polynomial trace from the selected element
$T_0$.
This convention determines which element supports the regularized
source in the next lemma.
\begin{lemma}[Regularized point functionals]
\label{lem:regularized-point}
Let $x_0\in\overline{T_0}$, let $E\in\mathbb S$ be a unit tensor, and let
$\nu\in\R^2$ be a unit direction.  There exist $ \delta_{h,x_0}^{E},
 \delta_{h,x_0}^{E,\nu}
 \in C_0^\infty(T_0;\mathbb S)$
such that
\begin{align*}
 (\delta_{h,x_0}^{E},\tau_h)_{T_0}
 &=
 E:\tau_h(x_0),
 &&\tau_h\in\Sigma_h|_{T_0},\\
 (\delta_{h,x_0}^{E,\nu},\tau_h)_{T_0}
 &=
 E:\partial_\nu(\tau_h|_{T_0})(x_0),
 &&\tau_h\in\Sigma_h|_{T_0},\quad p\ge1.
\end{align*}
The source norms satisfy
\[
 \|\delta_{h,x_0}^{E}\|_{L^1(T_0)}
 \lesssim 1,
 \qquad
 \|\delta_{h,x_0}^{E,\nu}\|_{L^1(T_0)}
 \lesssim h^{-1}.
\]
For $p\ge1$ we choose
$\delta_{h,x_0}^{E,\nu}:=-\partial_\nu\delta_{h,x_0}^{E}$.
The scaled $L^2$ bounds are
\begin{equation*}
 \|\delta_{h,x_0}^{E}\|_{L^2(T_0)}\lesssim h^{-1},
 \qquad
 \|\delta_{h,x_0}^{E,\nu}\|_{L^2(T_0)}\lesssim h^{-2}.
\end{equation*}
The two regularized sources may be chosen so that
\begin{align}
 \int_{T_0}\delta_{h,x_0}^{E}\,dx&=E,
 &
 \int_{T_0}|x-x_0|\,|\delta_{h,x_0}^{E}(x)|\,dx&\lesssim h,
 \label{eq:regularized-value-moments}
 \\
 \int_{T_0}\delta_{h,x_0}^{E,\nu}\,dx&=0,
 &
 \int_{T_0}|x-x_0|^j|\delta_{h,x_0}^{E,\nu}(x)|\,dx&\lesssim h^{j-1},
 \quad j=1,2.
 \notag
\end{align}
The constant is uniform with respect to $h$, $T_0$, and $x_0$.
\end{lemma}

\begin{proof}
Choose a smooth reference-element representer $\widehat\delta^E$
for polynomial point evaluation, with bounds uniform in the reference
point, and scale it affinely to $T_0$.
The uniformity follows from finite dimensionality on the reference
element and compactness of $\overline{\widehat T}$.
For $p\ge1$, set
$\delta_{h,x_0}^{E,\nu}=-\partial_\nu\delta_{h,x_0}^{E}$.
Integration by parts gives derivative reproduction and zero mean.
Affine scaling gives all $L^1$, $L^2$, and moment bounds, including
\eqref{eq:regularized-value-moments}, uniformly in $T_0$ and $x_0$.
\end{proof}

For $m\in\{0,1\}$, we write
\[
 \delta_{h,x_0}^{(0)}:=\delta_{h,x_0}^{E},
 \qquad
 \delta_{h,x_0}^{(1)}:=\delta_{h,x_0}^{E,\nu},
 \qquad
 L_m(\tau):=(\delta_{h,x_0}^{(m)},\tau)_U.
\]
The derivative source and $L_1$ are used only when
$p\ge1$.  These functionals define the local Green stresses used below.

For the main theorem, $\mathcal K$ is constant.  Set
$\mathcal C:=\mathcal K^{-1}$ and define $z^{(m)}\in H_0^2(U)$ by
\begin{equation}
 (\mathcal C\nabla^2z^{(m)},\nabla^2v)_U
 =-(\mathcal C\delta_{h,x_0}^{(m)},\nabla^2v)_U
 \qquad\forall v\in H_0^2(U),
 \label{eq:HHJ-clamped-Green-problem}
\end{equation}
and define the HHJ Green stress
\begin{equation}
 G^{(m)}
 :=\mathcal C\bigl(\nabla^2z^{(m)}
                  +\delta_{h,x_0}^{(m)}\bigr).
 \label{eq:HHJ-clamped-Green-stress}
\end{equation}
The space $H_0^2(U)$ imposes the clamped traces
$z^{(m)}=\partial_n z^{(m)}=0$ on $\partial U$.

For every $\tau\in L^2(U;\mathbb S)$,
\eqref{eq:HHJ-clamped-Green-stress} gives
\begin{equation}
 (\mathcal K G^{(m)},\tau)_U
 -(\tau,\nabla^2z^{(m)})_U
 =
 (\delta_{h,x_0}^{(m)},\tau)_U.
 \label{eq:HHJ-Green-mixed-identity}
\end{equation}
Moreover, \eqref{eq:HHJ-clamped-Green-problem} implies
\begin{equation}
 (G^{(m)},\nabla^2v)_U=0
 \qquad\forall v\in H_0^2(U).
 \label{eq:HHJ-Green-kernel-identity}
\end{equation}

\begin{lemma}[Interior HHJ Green-stress estimates]
\label{lem:HHJ-interior-Green-stresses}
Let $U=B_{5d/2}(x_0)$, and let $G^{(m)},z^{(m)}$ be defined
by \eqref{eq:HHJ-clamped-Green-problem}--
\eqref{eq:HHJ-clamped-Green-stress}.  Then
\begin{align*}
 \|G^{(0)}\|_{L^2(U)}
 &\lesssim h^{-1},
 &
 \|G^{(1)}\|_{L^2(U)}
 &\lesssim h^{-2}.
\end{align*}
For every interior annular patch $D_\delta^+$ satisfying
\[
 \kappa_0h\le\delta\le cd,
 \qquad
 D_\delta^+\Subset B_{2d}(x_0),
 \qquad
 \dist(D_\delta^+,x_0)\simeq\delta,
\]
one has
\begin{align*}
 \|G^{(0)}\|_{L^2(D_\delta^+)}
 +\delta\|\nabla G^{(0)}\|_{L^2(D_\delta^+)}
 &\lesssim \delta^{-1},
 \\
 \|G^{(1)}\|_{L^2(D_\delta^+)}
 +\delta\|\nabla G^{(1)}\|_{L^2(D_\delta^+)}
 &\lesssim \delta^{-2}.
\end{align*}
\end{lemma}

\begin{proof}
Lax--Milgram gives a unique solution of
\eqref{eq:HHJ-clamped-Green-problem}; testing with $v=z^{(m)}$ yields
\[
 \|\nabla^2z^{(m)}\|_{L^2(U)}
 \lesssim
 \|\delta_{h,x_0}^{(m)}\|_{L^2(U)}.
\]
Hence the stated $L^2(U)$ bounds follow from
\eqref{eq:HHJ-clamped-Green-stress} and
Lemma~\ref{lem:regularized-point}.

Let $\mathscr G_U(x,y)$ be the clamped Green kernel of
$\DivDiv(\mathcal C\nabla^2\cdot)$ on the ball $U$.  The interior
Green-function estimates for elliptic boundary problems
\cite{AgmonDouglisNirenberg1959,AgmonDouglisNirenberg1964,Krasovskii1967}
imply, for $|x-y|\simeq|x-x_0|\ge\kappa_0 h$,
\begin{equation*}
 |D_x^\mu D_y^\nu\mathscr G_U(x,y)|
 \lesssim
 |x-y|^{2-|\mu|-|\nu|},
 \qquad |\mu|+|\nu|\ge3.
\end{equation*}
Here the source has diameter $O(h)$, while $B_{2d}(x_0)$ is separated
from $\partial U$ by a fixed fraction of $d$.

Outside the source element, Green representation for
\eqref{eq:HHJ-clamped-Green-problem} gives, with repeated indices summed,
\begin{equation*}
 (G^{(0)})_{ij}(x)
 =
 -\mathcal C_{ijmn}
 \int_{T_0}
 \partial_{x_mx_n}\partial_{y_ky_l}\mathscr G_U(x,y)
 (\mathcal C\delta_{h,x_0}^{E}(y))_{kl}\,dy.
\end{equation*}
Combining this kernel bound with
\eqref{eq:regularized-value-moments} yields
\[
 |G^{(0)}(x)|
 \lesssim |x-x_0|^{-2},
 \qquad
 |\nabla G^{(0)}(x)|
 \lesssim |x-x_0|^{-3}.
\]
For the derivative source,
$\delta_{h,x_0}^{E,\nu}=-\partial_\nu\delta_{h,x_0}^{E}$.
Integration by parts in $y$ gives
\[
 |G^{(1)}(x)|
 \lesssim |x-x_0|^{-3},
 \qquad
 |\nabla G^{(1)}(x)|
 \lesssim |x-x_0|^{-4}.
\]
Integration over annuli of radius $\delta$ gives the stated bounds.
\end{proof}

The kernel identity also connects the Green stress to symmetric-curl
test fields.  Indeed, for every
$\psi\in H_0^1(U;\mathbb R^2)$ with compact support in $U$, the mixed
identity and $\DivDiv(\symcurl\psi)=0$ imply
\begin{equation*}
 a_U(G^{(m)},\symcurl\psi)
 =
 L_m(\symcurl\psi).
\end{equation*}

\begin{lemma}[Interior point representation and boundary cancellation]
Let $\xi_h\in Z_h$ and let $\widetilde z^{(m)}$ be the zero extension
of $z^{(m)}$ from $U$ to $\Omega$.  Then
\begin{equation*}
 L_m(\xi_h)
 =a_U(G^{(m)},\xi_h).
\end{equation*}
\end{lemma}

\begin{proof}
Since $z^{(m)}\in H_0^2(U)$ and $U$ is a smooth ball,
$\widetilde z^{(m)}\in H_0^2(\Omega)$.  The second Fortin identity
and $\xi_h\in Z_h$ give
$b_h(\xi_h,\widetilde z^{(m)})
 =b_h(\xi_h,I_h\widetilde z^{(m)})=0$.

Since $\widetilde z^{(m)}$ and its normal derivative vanish on
the artificial boundary, its broken integration-by-parts formula has
no contribution from $\partial U$; hence
\[
 (\xi_h,\nabla^2z^{(m)})_U
 =(\xi_h,\nabla^2\widetilde z^{(m)})_\Omega=0.
\]
Taking $\tau=\xi_h|_U$ in
\eqref{eq:HHJ-Green-mixed-identity} proves the claim.
\end{proof}

\begin{lemma}[Potential representatives of the local Green stresses]
\label{lem:Green-potential-representatives}
For every $m\in\{0,1\}$, there exists $ \Phi^{(m)}
 \in
 H^1(U;\mathbb R^2)$ such that $ \symcurl\Phi^{(m)}
 =
 G^{(m)}$.  The representative is determined only modulo
$\mathcal R(U)$.
\end{lemma}

\begin{proof}
By \eqref{eq:HHJ-Green-kernel-identity},
$\DivDiv G^{(m)}=0$ in $\mathcal D'(U)$.  Since $U$ is contractible,
exactness of the continuous symmetric-curl complex
\cite{ArnoldFalkWinther2010,BoffiBrezziFortin} gives
$G^{(m)}=\symcurl\Phi^{(m)}$, uniquely modulo $\mathcal R(U)$.
\end{proof}

\subsection{Kernel-preserving potential localization}

Choose $\chi_0,\chi\in C_0^\infty(U)$ with
\begin{equation*}
\begin{gathered}
 0\le \chi,\chi_0\le1,\qquad
 \chi=1\ \text{on }B_d(x_0),
 \\
 \supp\chi\subset B_{3d/2}(x_0),\qquad
 \chi_0=1\ \text{on }B_{2d}(x_0),
 \\
 \supp\chi_0\subset B_{9d/4}(x_0),
 \\
 \|D^j\chi\|_{L^\infty(U)}
 +\|D^j\chi_0\|_{L^\infty(U)}\lesssim d^{-j},
 \qquad 0\le j\le2.
\end{gathered}
\end{equation*}
Choose $\kappa_0$ in $d\ge\kappa_0 h$ large enough that both
the following nesting and the transition-patch separation below hold:
\begin{equation}
 (\supp\chi)^+
 \Subset
 B_{7d/4}(x_0)
 \Subset
 \{\chi_0=1\}
 \Subset
 \supp\chi_0
 \Subset
 U.
\label{eq:nested-cutoff-separation}
\end{equation}

Set $S_\chi:=\supp(\nabla\chi)$ and define the mesh-fitted transition
patch
\[
 \omega_\chi:=\operatorname{int}\bigcup
 \{T\in\mathcal T_h:T\cap S_\chi\ne\varnothing\}.
\]
Choose fixed-layer enlargements $\omega_\chi^+$ and
$\omega_\chi^{++}$ satisfying
\[
 S_\chi\subset\omega_\chi\Subset\omega_\chi^{++}
 \Subset\{\chi_0=1\}\Subset U.
\]
These patches belong to a connected, uniformly Lipschitz family of
mesh patches specified above.  To measure the potential on such a
transition patch,
for an open set $D\subset U$ and a finite-dimensional Euclidean space
$X$, define the local negative norm by
\begin{equation*}
 \|q\|_{H^{-1}(D;X)}
 :=
 \sup_{\substack{\psi\in H_0^1(D;X)\\ \psi\ne0}}
 \frac{|(q,\psi)_D|}
      {\|\nabla\psi\|_{L^2(D)}}.
\end{equation*}

\begin{lemma}[Potential control by a local negative norm]
\label{lem:potential-negative-norm}
Let $D$ belong to a connected, uniformly Lipschitz family of
mesh-fitted patches.
Then
\begin{equation*}
 \inf_{r\in\mathcal R(D)}
 \|\phi-r\|_{L^2(D)}
 \lesssim
 \|\symcurl\phi\|_{H^{-1}(D;\mathbb S)},
 \qquad
 \phi\in H^1(D;\mathbb R^2).
\end{equation*}
The constant is uniform over this family of mesh patches.
\end{lemma}

\begin{proof}
Let
\[
J_\perp=
\begin{pmatrix}
0&-1\\
1&0
\end{pmatrix},
\qquad
\widehat D:=J_\perp^TD.
\]
For $y\in\widehat D$, set $u(y):=\phi(J_\perp y)$.  Then
\[
 \nabla_yu(y)=\nabla\phi(J_\perp y)J_\perp,
 \qquad
 \operatorname{sym}\nabla_yu(y)=\symcurl\phi(J_\perp y).
\]
Rotation maps the rigid-motion kernel of $\operatorname{sym}\nabla$
onto $\mathcal R(D)=\ker(\symcurl)$.

The negative-norm Korn inequality
\cite{HlavacekNecas1970,LewintanNeff2021} gives
\[
\inf_{q\in\operatorname{RM}(\widehat D)}
\|u-q\|_{L^2(\widehat D)}
 \lesssim
 \|\operatorname{sym}\nabla u\|_{H^{-1}(\widehat D)}.
\]
Rotating back gives the claim.  Scaling from the reference Lipschitz
family makes the constant uniform.
\end{proof}

For $\widetilde\tau_h\in Z_h$, set
$\tau_h:=\widetilde\tau_h|_U$ and
$\phi_h:=\left.R_h\widetilde\tau_h\right|_U$.  Let
$r_\chi\in\mathcal R(\omega_\chi^{++})$ be an
$L^2(\omega_\chi^{++})$-best approximation to $\phi_h$.
For a function compactly supported in $U$, let $\operatorname{Ext}_0$
denote its extension by zero to $\Omega$.
Define the kernel-preserving localization operator by
\begin{equation}
 \mathscr C_h(\chi,\tau_h)
 :=
 \symcurl J_h\bigl(
   \operatorname{Ext}_0[\chi(\phi_h-r_\chi)]
 \bigr).
 \label{eq:HHJ-potential-localization-operator}
\end{equation}
We also set
\begin{align*}
 R_{\chi,h}(\tau_h)
 &:=
 \symcurl\!\left[
 J_h\bigl(\operatorname{Ext}_0[\chi(\phi_h-r_\chi)]\bigr)
 -
 \operatorname{Ext}_0[\chi(\phi_h-r_\chi)]
 \right],
 \\
 \nabla^\perp\chi
 &:=(\partial_2\chi,-\partial_1\chi).
\end{align*}
\begin{lemma}[Potential localization]
\label{lem:HHJ-symmetric-potential-commutator}
The operator $\mathscr C_h$ satisfies
\begin{equation*}
 \mathscr C_h(\chi,\tau_h)\in Z_h,
 \qquad
 \supp\mathscr C_h(\chi,\tau_h)
 \subset
 (\supp\chi)^+
 \Subset
 \{\chi_0=1\}
 \Subset U.
\end{equation*}
Its defect satisfies $ \supp\bigl(
 \mathscr C_h(\chi,\tau_h)-\chi\tau_h
 \bigr)
 \subset
 \omega_\chi^{++}$.

The localization defect admits the decomposition
\begin{equation*}
 \mathscr C_h(\chi,\tau_h)-\chi\tau_h
 =
 R_{\chi,h}(\tau_h)
 +
 \operatorname{sym}\bigl(
   (\phi_h-r_\chi)\otimes\nabla^\perp\chi
 \bigr),
\end{equation*}
with $ \|R_{\chi,h}(\tau_h)\|_{L^2(U)}
 \lesssim
 \frac hd
 \|\tau_h\|_{L^2(\omega_\chi^{++})}$ and
\begin{align*}
 \left\|
 \operatorname{sym}\bigl(
   (\phi_h-r_\chi)\otimes\nabla^\perp\chi
 \bigr)
 \right\|_{L^2(U)}
 \lesssim
 d^{-1}
 \|\phi_h-r_\chi\|_{L^2(\omega_\chi^{++})}.
\end{align*}
Thus the localization defect satisfies
\begin{equation*}
 \|\mathscr C_h(\chi,\tau_h)-\chi\tau_h\|_{L^2(U)}
 \lesssim
 \|\tau_h\|_{L^2(\omega_\chi^{++})}.
\end{equation*}
In terms of the global extension $\widetilde\tau_h\in Z_h$, the stronger
defect bound is
\begin{equation*}
 \|\mathscr C_h(\chi,\tau_h)-\chi\tau_h\|_{L^2(U)}
 \lesssim
 d^{-1}
 \|\widetilde\tau_h\|_{H^{-1}(\Omega;\mathbb S)}.
\end{equation*}
\end{lemma}

\begin{proof}
For sufficiently regular $v$,
\begin{equation*}
 \symcurl(\chi v)
 =
 \chi\symcurl v
 +
 \operatorname{sym}\bigl(
   v\otimes\nabla^\perp\chi
 \bigr).
\end{equation*}
With $v=\phi_h-r_\chi$, $\symcurl r_\chi=0$, and
$\symcurl\phi_h=\tau_h$, this gives
\[
 \symcurl\bigl(\chi(\phi_h-r_\chi)\bigr)
 =\chi\tau_h
 +\operatorname{sym}\bigl(
   (\phi_h-r_\chi)\otimes\nabla^\perp\chi
 \bigr).
\]
Together with \eqref{eq:HHJ-potential-localization-operator}, this gives
the defect decomposition.

The local superapproximation property of $J_h$, together with the
scaled Korn inequality modulo $\mathcal R(\omega_\chi^{++})$, yields
\begin{align*}
 \|R_{\chi,h}(\tau_h)\|_{L^2(U)}
 \lesssim
 \frac hd
 \|\nabla(\phi_h-r_\chi)\|_{L^2(\omega_\chi^{++})}
 +\frac h{d^2}
 \|\phi_h-r_\chi\|_{L^2(\omega_\chi^{++})}
\lesssim
 \frac hd
 \|\tau_h\|_{L^2(\omega_\chi^{++})}.
\end{align*}
Since $ \|\nabla\chi\|_{L^\infty(U)}
 \lesssim d^{-1}$, we have
\[
 \left\|
 \operatorname{sym}\bigl(
   (\phi_h-r_\chi)\otimes\nabla^\perp\chi
 \bigr)
 \right\|_{L^2(U)}
 \lesssim
 d^{-1}
 \|\phi_h-r_\chi\|_{L^2(\omega_\chi^{++})}.
\]
Moreover, the scaled Korn inequality modulo
$\mathcal R(\omega_\chi^{++})$ implies
\[
 d^{-1}
 \|\phi_h-r_\chi\|_{L^2(\omega_\chi^{++})}
 \lesssim
 \|\tau_h\|_{L^2(\omega_\chi^{++})}.
\]
These estimates give the $L^2$ defect bound.  To obtain the stronger
negative-norm bound, Lemma~\ref{lem:potential-negative-norm} gives
\[
 \|\phi_h-r_\chi\|_{L^2(\omega_\chi^{++})}
 \lesssim
 \|\tau_h\|_{H^{-1}(\omega_\chi^{++};\mathbb S)}
 \lesssim
 \|\widetilde\tau_h\|_{H^{-1}(\Omega;\mathbb S)}.
\]
For each $T\subset\omega_\chi^{++}$, let $b_T$ be the element bubble
and define, by zero extension,
\[
 \psi
 :=
 \sum_{T\subset\omega_\chi^{++}}
 b_T\,\tau_h|_T
 \in H_0^1(\Omega;\mathbb S).
\]
Polynomial norm equivalence and the inverse estimate give
\[
 (\widetilde\tau_h,\psi)_\Omega
 \gtrsim
 \|\tau_h\|_{L^2(\omega_\chi^{++})}^2,
 \qquad
 \|\nabla\psi\|_{L^2(\Omega)}
 \lesssim
 h^{-1}\|\tau_h\|_{L^2(\omega_\chi^{++})}.
\]
Hence $h\|\tau_h\|_{L^2(\omega_\chi^{++})}
 \lesssim
 \|\widetilde\tau_h\|_{H^{-1}(\Omega;\mathbb S)}$.  Substitution in
the defect decomposition gives the negative-norm bound.

Locality of $J_h$ and
\eqref{eq:nested-cutoff-separation} imply $ \supp\mathscr C_h(\chi,\tau_h)
 \subset
 (\supp\chi)^+
 \Subset
 \{\chi_0=1\}$.
Since $\DivDiv\symcurl=0$, the localized stress belongs to $Z_h$.
Outside $\omega_\chi^{++}$, $\chi$ is locally constant and
$\phi_h-r_\chi$ is a finite-element function, so locality and the
projection property of $J_h$ give
$\mathscr C_h(\chi,\tau_h)=\chi\tau_h$.
Hence the defect is supported in $\omega_\chi^{++}$.
\end{proof}

\section{Localized maximum-norm estimates}

Throughout this section, the triangulation remains shape-regular and
quasi-uniform.  We use the local constructions of
Section~\ref{sec:HHJ-framework-local-Green} to connect point evaluation
to the global discrete problem.

\subsection{The global discrete Green stress}
\leavevmode\par

Assume that $\mathcal K$ is constant.  For $F\in L^2(\Omega;\mathbb R^2)$, let
$\zeta_F\in H_0^1(\Omega;\mathbb R^2)$ satisfy
\begin{equation}
 \mathcal A_\Omega(\zeta_F,\psi)=(F,\psi)_\Omega
 \qquad
 \forall\psi\in H_0^1(\Omega;\mathbb R^2).
 \label{eq:HHJ-potential-Dirichlet-dual}
\end{equation}
We assume the analytic potential $H^2$ shift associated with
\eqref{eq:HHJ-potential-Dirichlet-dual}
\cite{AgmonDouglisNirenberg1959,Grisvard1985,Dauge1988}:
\begin{equation}
 \|\zeta_F\|_{H^2(\Omega)}
 \le C\|F\|_{L^2(\Omega)}.
 \label{eq:HHJ-potential-H2-shift}
\end{equation}
For the assertions restricted to $p\ge1$, we also assume the analytic
potential $H^3$ shift
\cite{AgmonDouglisNirenberg1959,Grisvard1985,Dauge1988}:
\begin{equation}
 \|\zeta_F\|_{H^3(\Omega)}
 \le C\|F\|_{H^1(\Omega)}
 \qquad
 \forall F\in H_0^1(\Omega;\mathbb R^2).
 \label{eq:HHJ-potential-H3-shift}
\end{equation}

For $m\in\{0,1\}$, let
$\Phi^{(m)}$ be the representative from
Lemma~\ref{lem:Green-potential-representatives}, and choose
$r_0^{(m)}\in\mathcal R(U)$ so that
\begin{equation}
 \|\Phi^{(m)}-r_0^{(m)}\|_{L^2(U)}
 \lesssim d\|G^{(m)}\|_{L^2(U)},
 \label{eq:HHJ-Green-potential-normalization}
\end{equation}
and define $v^{(m)}
 :=
 \chi_0(\Phi^{(m)}-r_0^{(m)})
 \in H_0^1(\Omega;\mathbb R^2)$ after extension by zero outside $U$.

Let $\Phi_h^{(m)}\in W_{h,0}$ be the global Galerkin
projection determined by
\begin{equation}
 \mathcal A_\Omega(\Phi_h^{(m)},\psi_h)
 =
 \mathcal A_\Omega(v^{(m)},\psi_h)
 \qquad
 \forall\psi_h\in W_{h,0},
 \label{eq:HHJ-potential-Galerkin-Green}
\end{equation}
and set
\begin{equation}
 G_h^{(m)}
 :=
 \symcurl\Phi_h^{(m)}.
 \label{eq:HHJ-potential-Galerkin-Green-stress}
\end{equation}
Both $\Phi_h^{(m)}$ and $G_h^{(m)}$ are global
fields, and their approximation defect is denoted by
$\varepsilon_h^{(m)}:=v^{(m)}-\Phi_h^{(m)}$.

\begin{lemma}[Global discrete Green-stress estimates]
For $m\in\{0,1\}$,
$G_h^{(m)}\in Z_h$.  If $\psi_h\in W_{h,0}$ and
$\supp\psi_h\Subset\{\chi_0=1\}$, then
\begin{equation}
 a_\Omega(G_h^{(m)},\symcurl\psi_h)
 =
 L_m(\symcurl\psi_h).
 \label{eq:HHJ-potential-Green-source-identity}
\end{equation}

The following estimates also hold:
\begin{equation}
 \|\symcurl\varepsilon_h^{(0)}\|_{L^2(\Omega)}
 \lesssim h^{-1}.
 \label{eq:HHJ-value-global-Galerkin-energy}
\end{equation}
\begin{equation*}
 \|\symcurl\varepsilon_h^{(1)}\|_{L^2(\Omega)}
 \lesssim h^{-2}.
\end{equation*}
\begin{equation}
 \|\varepsilon_h^{(0)}\|_{L^2(\Omega)}
 \lesssim 1.
 \label{eq:HHJ-value-global-Galerkin-L2}
\end{equation}
Under \eqref{eq:HHJ-potential-H3-shift}, for $p\ge1$ one also has
\begin{align}
 \|\varepsilon_h^{(0)}\|_{H^{-1}(\Omega;\mathbb R^2)}
 &\lesssim h,
 \label{eq:HHJ-value-global-Galerkin-negative}\\
 \|\varepsilon_h^{(1)}\|_{H^{-1}(\Omega;\mathbb R^2)}
 &\lesssim 1.
 \label{eq:HHJ-derivative-global-Galerkin-negative}
\end{align}
\end{lemma}

\begin{proof}
The normalization \eqref{eq:HHJ-Green-potential-normalization} follows
from Korn's inequality modulo $\mathcal R(U)$.  Korn's inequality on
$H_0^1(\Omega;\mathbb R^2)$ and the positive definiteness of
$\mathcal K$ imply coercivity and unique solvability of
\eqref{eq:HHJ-potential-Galerkin-Green}.
The global HHJ complex then implies
\[
 G_h^{(m)}
 =\symcurl\Phi_h^{(m)}\in Z_h.
\]

Let $\psi_h$ satisfy the support condition in the statement.  On a
neighborhood of its support,
$\symcurl v^{(m)}=\symcurl\Phi^{(m)}=G^{(m)}$.  The global Galerkin
equation and the local Green identity yield
\begin{align*}
 a_\Omega(G_h^{(m)},\symcurl\psi_h)
 &=\mathcal A_\Omega(v^{(m)},\psi_h)
 =a_U(G^{(m)},\symcurl\psi_h)
 =L_m(\symcurl\psi_h).
\end{align*}

The Galerkin error satisfies
\begin{equation*}
 \mathcal A_\Omega(\varepsilon_h^{(m)},\psi_h)=0
 \qquad
 \forall\psi_h\in W_{h,0}.
\end{equation*}
By Galerkin stability, the product rule for $\symcurl$, and
\eqref{eq:HHJ-Green-potential-normalization}, we have
\begin{align*}
 \|\symcurl\varepsilon_h^{(m)}\|_{L^2(\Omega)}
 &\lesssim \|\symcurl v^{(m)}\|_{L^2(U)}\\
 &\lesssim \left(
   \|G^{(m)}\|_{L^2(U)}
   +d^{-1}\|\Phi^{(m)}-r_0^{(m)}\|_{L^2(U)}
 \right)
 \lesssim \|G^{(m)}\|_{L^2(U)}.
\end{align*}
Lemma~\ref{lem:HHJ-interior-Green-stresses} now gives the first two
energy estimates in the statement.

For the weak-norm bounds, let $F\in L^2(\Omega;\mathbb R^2)$
and let $\zeta_F$ solve
\eqref{eq:HHJ-potential-Dirichlet-dual}.  Galerkin orthogonality, the
degree-$(p+1)$ potential approximation, and
\eqref{eq:HHJ-potential-H2-shift} yield
\begin{align*}
 |(\varepsilon_h^{(m)},F)_\Omega|
 &=|\mathcal A_\Omega(\varepsilon_h^{(m)},
          \zeta_F-J_h\zeta_F)|
  \lesssim h
 \|\symcurl\varepsilon_h^{(m)}\|_{L^2(\Omega)}
 \|F\|_{L^2(\Omega)}.
\end{align*}
Taking $m=0$ and using
\eqref{eq:HHJ-value-global-Galerkin-energy} gives
\eqref{eq:HHJ-value-global-Galerkin-L2}. For $p\ge1$ and
$F\in H_0^1(\Omega;\mathbb R^2)$, quadratic potential approximation
and \eqref{eq:HHJ-potential-H3-shift} instead yield
\begin{align*}
 |(\varepsilon_h^{(m)},F)_\Omega|
  &\lesssim h^2
 \|\symcurl\varepsilon_h^{(m)}\|_{L^2(\Omega)}
 \|F\|_{H^1(\Omega)}.
\end{align*}
Taking the supremum over $F$ and using
the two energy estimates above leads to
\eqref{eq:HHJ-value-global-Galerkin-negative} and
\eqref{eq:HHJ-derivative-global-Galerkin-negative}.
\end{proof}

\begin{lemma}[Interior Galerkin estimate for the potential problem]
\label{lem:HHJ-interior-potential-Galerkin}
Let $D\Subset D^+\Subset\Omega$ be mesh-fitted patches with
\begin{equation*}
 \dist(D,\partial D^+)\simeq\delta,
 \qquad
 \delta\ge\kappa_0h.
\end{equation*}
Let $v\in H_0^1(\Omega;\mathbb R^2)$ with
$v|_{D^+}\in H^2(D^+;\mathbb R^2)$, and let
$v_h\in W_{h,0}$ be its Galerkin projection with respect to
$\mathcal A_\Omega$.  For $e_h:=v-v_h$ and $s\in\{0,1\}$,
\begin{equation*}
 \|\symcurl e_h\|_{L^2(D)}
 \lesssim
 h|v|_{H^2(D^+)}
 +\delta^{-1-s}
 \|e_h\|_{H^{-s}(\Omega;\mathbb R^2)}.
\end{equation*}
The case $s=1$ is used below only with the corresponding global
$H^{-1}$ estimate.
\end{lemma}

\begin{proof}
Since $\mathcal K$ is constant, $\mathcal A_\Omega$ is a
constant-coefficient second-order system.  Its principal symbol satisfies
\[
 |\symcurl_\xi z|^2\ge c|\xi|^2|z|^2,
 \qquad \xi\ne0,
\]
so the system is strongly elliptic.  Korn's inequality on
$H_0^1(\Omega;\mathbb R^2)$ gives
$\|\nabla w\|_{L^2(\Omega)}\simeq
 \|\symcurl w\|_{L^2(\Omega)}$; hence the affine kernel causes no
coercivity obstruction for the global Galerkin problem.

The interior Galerkin estimate for conforming strongly elliptic systems
\cite{NitscheSchatz1974,SchatzWahlbin1995}, applied to this vector
problem, gives
\[
 \|\nabla e_h\|_{L^2(D)}
 \lesssim
 \inf_{w_h\in W_{h,0}}
 \|\nabla(v-w_h)\|_{L^2(D^+)}
 +\delta^{-1-s}
 \|e_h\|_{H^{-s}(\Omega;\mathbb R^2)},
 \qquad s=0,1.
\]
The powers of $\delta$ are those of the scaled negative-norm interior
estimate: one inverse length accompanies the $L^2$ pollution term, and
passing from $L^2$ to $H^{-1}$ contributes one further inverse length.
The local approximation property of $J_h$ bounds the first term by
$h|v|_{H^2(D^+)}$.  Finally,
$\|\symcurl e_h\|_{L^2(D)}\lesssim
 \|\nabla e_h\|_{L^2(D)}$, which proves the claimed estimate.
\end{proof}

\begin{lemma}[Local discrete Green estimates]
\label{lem:HHJ-local-discrete-Green-estimates}
Let $D_\delta$ be an annular mesh patch separated from the source
element by a distance comparable to $\delta$, where
$\kappa_0 h\le\delta\le Cd$, and let $D_\delta^+$ be a fixed enlargement
such that
\begin{equation*}
 D_\delta\Subset D_\delta^+
 \Subset\{\chi_0=1\},
 \qquad
 \dist(D_\delta,\partial D_\delta^+)\ge c\delta.
\end{equation*}
The following estimates hold:
\begin{align*}
 \|G_h^{(0)}\|_{L^2(D_\delta)}
 +
 \delta^{-1}
 \inf_{r\in\mathcal R(D_\delta)}
 \|\Phi_h^{(0)}-r\|_{L^2(D_\delta)}
 &\lesssim \delta^{-1},
 \\
 \|G_h^{(1)}\|_{L^2(D_\delta)}
 +
 \delta^{-1}
 \inf_{r\in\mathcal R(D_\delta)}
 \|\Phi_h^{(1)}-r\|_{L^2(D_\delta)}
 &\lesssim \delta^{-2},
 \qquad p\ge1.
\end{align*}
For the value Galerkin error, one also has
\begin{equation}
 \|\symcurl\varepsilon_h^{(0)}\|_{L^2(D_\delta)}
 \lesssim h\delta^{-2},
 \qquad p\ge1.
 \label{eq:HHJ-value-annular-Galerkin-error}
\end{equation}
The same scale bounds for
$G_h^{(0)}$ and
$G_h^{(1)}$
hold on the source patch
$D_0:=B_{\kappa_0 h}(x_0)\cap U$,
with $\delta$ replaced by $h$.
\end{lemma}

\begin{proof}
For every patch $D\Subset D^+$ in a connected, uniformly Lipschitz
family and every $\Phi\in H^2(D^+;\mathbb R^2)$, scaling the local
second-order estimate yields
\begin{equation}
 |\Phi|_{H^2(D)}
 \lesssim \|\nabla(\symcurl\Phi)\|_{L^2(D^+)}.
 \label{eq:local-higher-order-symcurl-Korn}
\end{equation}
Choose a fixed intermediate patch
$D_\delta\Subset D_\delta'\Subset D_\delta^+$ with both separations
comparable to $\delta$.  Apply
Lemma~\ref{lem:HHJ-interior-potential-Galerkin} to
$e_h=\varepsilon_h^{(m)}$ on $D_\delta\Subset D_\delta'$ and apply
\eqref{eq:local-higher-order-symcurl-Korn} on
$D_\delta'\Subset D_\delta^+$.  Since
$v^{(m)}=\Phi^{(m)}-r_0^{(m)}$ on $D_\delta^+$ and affine fields do
not contribute to the $H^2$ seminorm, it follows that
\[
 h|v^{(m)}|_{H^2(D_\delta')}
 \lesssim
 h\|\nabla G^{(m)}\|_{L^2(D_\delta^+)}.
\]
Consequently, for $s=0,1$,
\begin{equation*}
 \|\symcurl\varepsilon_h^{(m)}\|_{L^2(D_\delta)}
 \lesssim
 h\|\nabla G^{(m)}\|_{L^2(D_\delta^+)}
 +\delta^{-1-s}
 \|\varepsilon_h^{(m)}\|_{H^{-s}(\Omega;\mathbb R^2)}.
\end{equation*}
Here $s=1$ is used only for $p\ge1$.

It follows from Lemma~\ref{lem:HHJ-interior-Green-stresses} that
\begin{align*}
 \|G^{(0)}\|_{L^2(D_\delta^+)}
 +\delta\|\nabla G^{(0)}\|_{L^2(D_\delta^+)}
 &\lesssim \delta^{-1},\\
 \|G^{(1)}\|_{L^2(D_\delta^+)}
 +\delta\|\nabla G^{(1)}\|_{L^2(D_\delta^+)}
 &\lesssim \delta^{-2}.
\end{align*}
Set $s=0$ in this interior estimate and combine it with
\eqref{eq:HHJ-value-global-Galerkin-L2} and $h\le C\delta$.  This gives
$\|\symcurl\varepsilon_h^{(0)}\|_{L^2(D_\delta)}
\lesssim\delta^{-1}$.  For $p\ge1$, setting $s=1$ and using
\eqref{eq:HHJ-value-global-Galerkin-negative} proves
\eqref{eq:HHJ-value-annular-Galerkin-error}.  The derivative estimate
\eqref{eq:HHJ-derivative-global-Galerkin-negative} similarly gives
\[
 \|\symcurl\varepsilon_h^{(1)}\|_{L^2(D_\delta)}
 \lesssim \delta^{-2}.
\]

On $D_\delta^+$,
$G_h^{(m)}=G^{(m)}-\symcurl\varepsilon_h^{(m)}$.
Thus the stress parts of the first two estimates follow.  On the
uniformly Lipschitz patch family, the resulting scaled estimate modulo
$\mathcal R(D_\delta)$ takes the form
\[
 \delta^{-1}
 \inf_{r\in\mathcal R(D_\delta)}
 \|\Phi_h^{(m)}-r\|_{L^2(D_\delta)}
 \lesssim
 \|G_h^{(m)}\|_{L^2(D_\delta)}.
\]
On $D_0:=B_{\kappa_0 h}(x_0)\cap U$, the global energy bounds and Korn's
inequality modulo $\mathcal R(D_0)$ give the same estimates with $\delta$
replaced by $h$, since $\kappa_0$ is fixed.
\end{proof}

Applying Lemma~\ref{lem:HHJ-local-discrete-Green-estimates} to
$\omega_\chi^{++}$ with $\delta=d$ gives
\begin{align}
 \|G_h^{(0)}\|_{L^2(\omega_\chi^{++})}
 +d^{-1}
 \inf_{r\in\mathcal R(\omega_\chi^{++})}
 \|\Phi_h^{(0)}-r\|_{L^2(\omega_\chi^{++})}
 &\lesssim d^{-1},
 \label{eq:HHJ-value-transition-discrete-Green}\\
 \|G_h^{(1)}\|_{L^2(\omega_\chi^{++})}
 +d^{-1}
 \inf_{r\in\mathcal R(\omega_\chi^{++})}
 \|\Phi_h^{(1)}-r\|_{L^2(\omega_\chi^{++})}
 &\lesssim d^{-2},
 \qquad p\ge1.
 \label{eq:HHJ-derivative-transition-discrete-Green}
\end{align}

\subsection{Localized discrete Green identity and commutator estimate}

\begin{lemma}[Localized discrete Green identity]
\label{lem:HHJ-localized-discrete-Green-identity}
For $m\in\{0,1\}$, the following identity
holds:
\begin{align*}
L_m(\xi_h)
={}
-a_U\!\left(
\rho_\sigma,
\mathscr C_h(\chi,G_h^{(m)})
\right)
+
a_U\!\left(
G_h^{(m)},
\mathscr C_h(\chi,\xi_h)
\right)
-
a_U\!\left(
\xi_h,
\mathscr C_h(\chi,G_h^{(m)})
\right).
\end{align*}
\end{lemma}

\begin{proof}
Let $\phi_h=\left.R_h\xi_h\right|_U$ and choose $r_\chi\in\mathcal R(\omega_\chi^{++})$ as in
Lemma~\ref{lem:HHJ-symmetric-potential-commutator}.
For $\kappa_0$ sufficiently large, the interpolation patch of the source
element $T_0$ is contained in the region where $\chi=1$.
The locality and projection properties of $J_h$ therefore imply
$\mathscr C_h(\chi,\xi_h)|_{T_0}=\xi_h|_{T_0}$.  Hence
\begin{equation*}
L_m(\xi_h)
=
L_m\bigl(\mathscr C_h(\chi,\xi_h)\bigr).
\end{equation*}

By locality of $J_h$ and \eqref{eq:nested-cutoff-separation},
\[
J_h\bigl(\operatorname{Ext}_0[\chi(\phi_h-r_\chi)]\bigr)
\in W_{h,0},
\qquad
\supp J_h\bigl(\operatorname{Ext}_0[\chi(\phi_h-r_\chi)]\bigr)
\Subset\{\chi_0=1\}.
\]
The source identity \eqref{eq:HHJ-potential-Green-source-identity} gives
\begin{align*}
L_m(\xi_h)
&=
a_\Omega\!\left(
G_h^{(m)},
\mathscr C_h(\chi,\xi_h)
\right)
=
a_U\!\left(
G_h^{(m)},
\mathscr C_h(\chi,\xi_h)
\right),
\end{align*}
where the second equality follows from
$\supp\mathscr C_h(\chi,\xi_h)\Subset U$.

Since $G_h^{(m)}\in Z_h$,
Lemma~\ref{lem:HHJ-symmetric-potential-commutator} implies
\[
\mathscr C_h(\chi,G_h^{(m)})
\in Z_h,
\qquad
\supp\mathscr C_h(\chi,G_h^{(m)})
\Subset U.
\]
Define the localized Green stress by
\[
 Q_h^{(m)}:=\mathscr C_h(\chi,G_h^{(m)}).
\]
Testing \eqref{eq:kernel-error} with $Q_h^{(m)}$ yields
\begin{align*}
a_U(\xi_h,Q_h^{(m)})
=
a_\Omega(\xi_h,Q_h^{(m)})
=
-a_\Omega(\rho_\sigma,Q_h^{(m)})
=
-a_U(\rho_\sigma,Q_h^{(m)}).
\end{align*}

Adding and subtracting $a_U\!\left(
\xi_h,
\mathscr C_h(\chi,G_h^{(m)})
\right)$ in the preceding point identity and using the tested error
equation proves the identity.
\end{proof}

\begin{lemma}[Commutator estimate]
\label{lem:HHJ-paired-potential-commutator}
Assume that $\mathcal K|_U\in W^{1,\infty}(U)$ with a bound uniform in
the localization scale.
Let
\[
\widetilde q_h,\widetilde\tau_h\in Z_h,
\qquad
q_h:=\widetilde q_h|_U,
\qquad
\tau_h:=\widetilde\tau_h|_U.
\]
The commutator satisfies
\begin{align*}
&
\left|
a_U\bigl(q_h,\mathscr C_h(\chi,\tau_h)\bigr)
-
a_U\bigl(\tau_h,\mathscr C_h(\chi,q_h)\bigr)
\right|
\\
&\qquad\lesssim
\left[
d^{-1}\|q_h\|_{L^2(\omega_\chi^{++})}
+
d^{-2}
\inf_{r\in\mathcal R(\omega_\chi^{++})}
\|R_h\widetilde q_h-r\|_{L^2(\omega_\chi^{++})}
\right]
\|\widetilde\tau_h\|_{H^{-1}(\Omega;\mathbb S)}.
\end{align*}
\end{lemma}

\begin{proof}
Set $\phi_h:=\left.R_h\widetilde\tau_h\right|_U$ and
$\psi_h:=\left.R_h\widetilde q_h\right|_U$.  Choose
$r_\tau,r_q\in\mathcal R(\omega_\chi^{++})$ as their respective best
quotient approximations.  By
Lemma~\ref{lem:HHJ-symmetric-potential-commutator},
\begin{align*}
\mathscr C_h(\chi,\tau_h)-\chi\tau_h
&=
R_{\chi,h}(\tau_h)+K_\chi(\tau_h),
\\
\mathscr C_h(\chi,q_h)-\chi q_h
&=
R_{\chi,h}(q_h)+K_\chi(q_h),
\end{align*}
where the two continuous commutator terms are
\begin{align*}
K_\chi(\tau_h)
&=
\operatorname{sym}
\bigl((\phi_h-r_\tau)\otimes\nabla^\perp\chi\bigr),
\\
K_\chi(q_h)
&=
\operatorname{sym}
\bigl((\psi_h-r_q)\otimes\nabla^\perp\chi\bigr).
\end{align*}

Since $a_U$ is symmetric,
$a_U(q_h,\chi\tau_h)=a_U(\tau_h,\chi q_h)$.  Hence,
\begin{align*}
&
a_U\bigl(q_h,\mathscr C_h(\chi,\tau_h)\bigr)
-
a_U\bigl(\tau_h,\mathscr C_h(\chi,q_h)\bigr)
\\
&\quad=
a_U\bigl(q_h,R_{\chi,h}(\tau_h)\bigr)
-
a_U\bigl(\tau_h,R_{\chi,h}(q_h)\bigr)
\\
&\qquad
+
a_U\bigl(q_h,K_\chi(\tau_h)\bigr)
-
a_U\bigl(\tau_h,K_\chi(q_h)\bigr).
\end{align*}

The superapproximation estimates imply
\begin{align*}
\|R_{\chi,h}(\tau_h)\|_{L^2(U)}
&\lesssim
\frac hd
\|\tau_h\|_{L^2(\omega_\chi^{++})},
\\
\|R_{\chi,h}(q_h)\|_{L^2(U)}
&\lesssim
\frac hd
\|q_h\|_{L^2(\omega_\chi^{++})}.
\end{align*}
The inverse negative-norm estimate on the transition patch yields
\[
h\|\tau_h\|_{L^2(\omega_\chi^{++})}
\lesssim
\|\widetilde\tau_h\|_{H^{-1}(\Omega;\mathbb S)}.
\]
Thus the two remainder terms satisfy
\begin{align*}
&
\left|
a_U\bigl(q_h,R_{\chi,h}(\tau_h)\bigr)
\right|
+
\left|
a_U\bigl(\tau_h,R_{\chi,h}(q_h)\bigr)
\right|
\lesssim
d^{-1}
\|q_h\|_{L^2(\omega_\chi^{++})}
\|\widetilde\tau_h\|_{H^{-1}(\Omega;\mathbb S)}.
\end{align*}

Lemma~\ref{lem:potential-negative-norm} implies
\[
 \|\phi_h-r_\tau\|_{L^2(\omega_\chi^{++})}
 \lesssim
 \|\widetilde\tau_h\|_{H^{-1}(\Omega;\mathbb S)}.
\]
Consequently, the first commutator term satisfies
\begin{align*}
\left|
a_U\bigl(q_h,K_\chi(\tau_h)\bigr)
\right|
\lesssim
d^{-1}
\|q_h\|_{L^2(\omega_\chi^{++})}
\|\widetilde\tau_h\|_{H^{-1}(\Omega;\mathbb S)}.
\end{align*}
The scaled Korn estimate modulo the affine kernel and the cutoff bounds imply
\begin{align*}
\|K_\chi(q_h)\|_{H^1(\omega_\chi^{++})}
\lesssim
\bigg[
d^{-1}\|q_h\|_{L^2(\omega_\chi^{++})}
+
d^{-2}
\|\psi_h-r_q\|_{L^2(\omega_\chi^{++})}
\bigg].
\end{align*}
Since $K_\chi(q_h)$ is supported strictly inside
$\omega_\chi^{++}$, the local $W^{1,\infty}$ regularity of $\mathcal K$
implies $\mathcal K K_\chi(q_h)
\in H_0^1(\omega_\chi^{++};\mathbb S)$ and
\begin{align*}
\|\mathcal K K_\chi(q_h)\|_{H^1(\omega_\chi^{++})}
\lesssim
\bigg[
d^{-1}\|q_h\|_{L^2(\omega_\chi^{++})}
+
d^{-2}
\|\psi_h-r_q\|_{L^2(\omega_\chi^{++})}
\bigg].
\end{align*}
Extending $\mathcal K K_\chi(q_h)$ by zero to $\Omega$ and using
$H^{-1}$--$H_0^1$ duality gives
\begin{align*}
\left|
a_U\bigl(\tau_h,K_\chi(q_h)\bigr)
\right|
\lesssim
\bigg[
d^{-1}\|q_h\|_{L^2(\omega_\chi^{++})}
+
d^{-2}
\|\psi_h-r_q\|_{L^2(\omega_\chi^{++})}
\bigg]
\|\widetilde\tau_h\|_{H^{-1}(\Omega;\mathbb S)}.
\end{align*}

Combining these bounds in the commutator decomposition
and minimizing over $r_q\in\mathcal R(\omega_\chi^{++})$ proves the claim.
\end{proof}

Apply Lemma~\ref{lem:HHJ-paired-potential-commutator} to
$\widetilde q_h=G_h^{(m)}$ and $\widetilde\tau_h=\xi_h$,
which belong to $Z_h$.  By \eqref{eq:HHJ-potential-Galerkin-Green-stress},
\[
\symcurl
\left(
\left.R_hG_h^{(m)}\right|_U
\right)
=
G_h^{(m)}|_U
=
\symcurl\bigl(\Phi_h^{(m)}|_U\bigr).
\]
Hence
$\left.R_hG_h^{(m)}\right|_U
-\Phi_h^{(m)}|_U\in\mathcal R(U)$.

Restricting to $\omega_\chi^{++}$ gives
\begin{align*}
&
\inf_{r\in\mathcal R(\omega_\chi^{++})}
\left\|
\left.R_hG_h^{(m)}\right|_U-r
\right\|_{L^2(\omega_\chi^{++})}
=
\inf_{r\in\mathcal R(\omega_\chi^{++})}
\left\|
\Phi_h^{(m)}-r
\right\|_{L^2(\omega_\chi^{++})}.
\end{align*}
Combining Lemma~\ref{lem:HHJ-paired-potential-commutator} with
Lemma~\ref{lem:HHJ-localized-discrete-Green-identity} yields
\begin{align*}
&
\left|
L_m(\xi_h)
+
a_U\bigl(
\rho_\sigma,
\mathscr C_h(\chi,G_h^{(m)})
\bigr)
\right|
\notag\\
&\qquad\lesssim
\left[
d^{-1}
\|G_h^{(m)}\|_{L^2(\omega_\chi^{++})}
+
d^{-2}
\inf_{r\in\mathcal R(\omega_\chi^{++})}
\|\Phi_h^{(m)}-r\|_
 {L^2(\omega_\chi^{++})}
\right]
\|\xi_h\|_{H^{-1}(\Omega;\mathbb S)}.
\end{align*}

Equations~\eqref{eq:HHJ-value-transition-discrete-Green}--
\eqref{eq:HHJ-derivative-transition-discrete-Green} yield
\begin{align}
\left|
L_0(\xi_h)
+
a_U\bigl(
\rho_\sigma,
\mathscr C_h(\chi,G_h^{(0)})
\bigr)
\right|
&\lesssim
d^{-2}
\|\xi_h\|_{H^{-1}(\Omega;\mathbb S)},
\label{eq:HHJ-value-localized-Green-bound}
\\
\left|
L_1(\xi_h)
+
a_U\bigl(
\rho_\sigma,
\mathscr C_h(\chi,G_h^{(1)})
\bigr)
\right|
&\lesssim
d^{-3}
\|\xi_h\|_{H^{-1}(\Omega;\mathbb S)},
\qquad p\ge1.
\label{eq:HHJ-derivative-localized-Green-bound}
\end{align}

\subsection{Dyadic estimates and the localized maximum-norm theorem}

Let
\begin{equation*}
\begin{aligned}
 d_j&:=2^j\kappa_0 h,
 \qquad
 A_0:=U\cap B_{d_0}(x_0),
 \\
 A_j&:=
 U\cap\bigl(
 B_{d_j}(x_0)\setminus B_{d_j/2}(x_0)
 \bigr),
 \qquad j\ge1.
\end{aligned}
\end{equation*}
Choose $N$ so that $d_N$ is comparable to the diameter of the
local Green domain.  Let $A_j^+$ be a fixed enlargement of $A_j$.

\begin{lemma}[Local interpolation estimate]
Assume the potential $H^2$ shift \eqref{eq:HHJ-potential-H2-shift}
and, for $p\ge1$, the $H^3$ shift \eqref{eq:HHJ-potential-H3-shift}.
For the value source,
\begin{equation}
\left|
a_U\bigl(
\rho_\sigma,
\mathscr C_h(\chi,G_h^{(0)})
\bigr)
\right|
\lesssim
\begin{cases}
\bigl(1+\log(d/h)\bigr)
\|\rho_\sigma\|_{L^\infty(B_{2d}(x_0);\Th)},
& p=0,
\\[1mm]
\|\rho_\sigma\|_{L^\infty(B_{2d}(x_0);\Th)},
& p\ge1
\end{cases}.
\label{eq:HHJ-value-interpolation-Green-pairing}
\end{equation}
For the derivative source and $p\ge1$,
\begin{equation}
\left|
a_U\bigl(
\rho_\sigma,
\mathscr C_h(\chi,G_h^{(1)})
\bigr)
\right|
\lesssim
h^{-1}
\|\rho_\sigma\|_{L^\infty(B_{2d}(x_0);\Th)}.
\label{eq:HHJ-derivative-interpolation-Green-pairing}
\end{equation}
\end{lemma}

\begin{proof}
Set $Q_h^{(m)}:=\mathscr C_h(\chi,G_h^{(m)})$.  The support
property in Lemma~\ref{lem:HHJ-symmetric-potential-commutator} implies
\begin{equation*}
 \supp Q_h^{(m)}\subset(\supp\chi)^+\Subset B_{2d}(x_0).
\end{equation*}

For every annulus $A_j$ meeting
$(\supp\chi)^+$, its fixed enlargement $A_j^+$ may be chosen so that
\[
A_j\Subset A_j^+\Subset\{\chi_0=1\},
\qquad
\dist(A_j,\partial A_j^+)\ge c d_j.
\]
Lemma~\ref{lem:HHJ-local-discrete-Green-estimates} with $D_\delta=A_j$
and $\delta=d_j$, together with the stability of $\mathscr C_h$, gives
\begin{align*}
\|Q_h^{(0)}\|_{L^2(A_j)}
&\lesssim d_j^{-1},
\\
\|Q_h^{(1)}\|_{L^2(A_j)}
&\lesssim d_j^{-2},
\qquad p\ge1.
\end{align*}
The source-patch estimate in
Lemma~\ref{lem:HHJ-local-discrete-Green-estimates} gives the same bounds
on $A_0$; outside $(\supp\chi)^+$, $Q_h^{(m)}=0$.

Since $|A_j|^{1/2}\lesssim d_j$, H\"older's inequality implies
\begin{align*}
 \|Q_h^{(0)}\|_{L^1(A_j)}&\lesssim1,
 &
 \|Q_h^{(1)}\|_{L^1(A_j)}&\lesssim d_j^{-1}.
\end{align*}

Summing over $j$, with $N+1\simeq1+\log(d/h)$, gives
\begin{align*}
\|Q_h^{(0)}\|_{L^1(U)}
&\lesssim
\bigl(1+\log(d/h)\bigr),
\\
\|Q_h^{(1)}\|_{L^1(U)}
&\lesssim
\sum_{j=0}^{N}d_j^{-1}
\lesssim h^{-1},
\qquad p\ge1.
\end{align*}
The boundedness of $\mathcal K$ converts these $L^1$ bounds into the
logarithmic value estimate for $p=0$ and the derivative estimate for
$p\ge1$.  The value estimate for $p\ge1$ requires an additional
cancellation to remove the logarithm.  For that estimate, let $p\ge1$
and set
$\varepsilon_h^{(0)}
:=v^{(0)}-\Phi_h^{(0)}$.  The global
negative-norm estimate \eqref{eq:HHJ-value-global-Galerkin-negative} implies
$\|\varepsilon_h^{(0)}\|_{H^{-1}(\Omega;\mathbb R^2)}
\lesssim h$.

Set
$\psi_h:=\left.R_hG_h^{(0)}\right|_U$.
Since
$\symcurl\psi_h=G_h^{(0)}|_U
=\symcurl(\Phi_h^{(0)}|_U)$, there exists
$s_U\in\mathcal R(U)$ such that
$\psi_h=\Phi_h^{(0)}+s_U$.

On $\supp\chi$, the definition of $v^{(0)}$ reads
\[
 \psi_h
 =\Phi^{(0)}-r_0^{(0)}
  -\varepsilon_h^{(0)}+s_U.
\]
Let $r_q\in\mathcal R(\omega_\chi^{++})$ be the best representative
of $\psi_h$ used in
$\mathscr C_h(\chi,G_h^{(0)})$, and let
$r_c\in\mathcal R(\omega_\chi^{++})$ be an $L^2$-best representative
of $\Phi^{(0)}$.  Define
\[
 r_e:=r_c-r_0^{(0)}+s_U-r_q
 \in\mathcal R(\omega_\chi^{++}).
\]
Then, on $\supp\chi$,
\[
 (\Phi^{(0)}-r_c)
 -(\varepsilon_h^{(0)}-r_e)
 =\psi_h-r_q.
\]
By linearity of $J_h$,
\begin{equation*}
Q_h^{(0)}
=
Q_h^{c}-Q_h^{e},
\end{equation*}
where the continuous and Galerkin-error parts are
\begin{align*}
Q_h^{c}
&:=
 \symcurl J_h
\bigl(
 \operatorname{Ext}_0[\chi(\Phi^{(0)}-r_c)]
\bigr),
\\
Q_h^{e}
&:=
 \symcurl J_h
\bigl(
 \operatorname{Ext}_0[\chi(\varepsilon_h^{(0)}-r_e)]
\bigr).
\end{align*}
This decomposition
removes the logarithm once its two parts are summed separately.  For the
continuous part, on the source elements whose interpolation patches are
contained in $\{\chi=1\}$, affine reproduction removes the choice of
$r_c$, and local stability of $J_h$ together with the $L^2$ bound for
the Green stress on the source patch gives
$\|Q_h^c\|_{L^2}\lesssim h^{-1}$.  Hence the elementwise inverse estimate
and $|A_0|^{1/2}\lesssim h$ yield
\begin{equation*}
 h\|\nabla_hQ_h^c\|_{L^1(A_0)}
 \lesssim
 \|Q_h^c\|_{L^1(A_0^+)}
 \lesssim1.
\end{equation*}
If a portion of $A_0$ meets the cutoff transition, which can occur only
when $d\simeq h$, that portion is included in the transition estimate
below.

If $A_j^+\subset\{\chi=1\}$, locality and affine reproduction of
$J_h$ allow the affine representative to be chosen on $A_j^+$.  The local $H^2$ stability
\eqref{eq:HHJ-Scott-Zhang-H2-stability} and
$\symcurl\Phi^{(0)}=G^{(0)}$ then give
\begin{equation*}
 \|\nabla_hQ_h^c\|_{L^2(A_j)}
 \lesssim
 \|\nabla G^{(0)}\|_{L^2(A_j^+)}
 \lesssim d_j^{-2}.
\end{equation*}
Since $|A_j|^{1/2}\lesssim d_j$, this estimate yields
$h\|\nabla_hQ_h^c\|_{L^1(A_j)}\lesssim h/d_j$.

It remains to consider the annuli meeting $\supp(\nabla\chi)$.  There
are only a uniformly bounded number of them and $d_j\simeq d$ there.
Using the fixed representative $r_c$ from the definition of $Q_h^c$,
the local $H^2$ stability
\eqref{eq:HHJ-Scott-Zhang-H2-stability} and the cutoff product rule give
\begin{align*}
 \|\nabla_hQ_h^c\|_{L^2(A_j)}
 \lesssim{}&
 |\Phi^{(0)}-r_c|_{H^2(\omega_\chi^{++})}
 +d^{-1}\|G^{(0)}\|_{L^2(\omega_\chi^{++})}
 +d^{-2}\|\Phi^{(0)}-r_c\|_{L^2(\omega_\chi^{++})}.
\end{align*}
Here $\|\nabla\chi\|_{L^\infty}\lesssim d^{-1}$ and
$\|D^2\chi\|_{L^\infty}\lesssim d^{-2}$.  The $H^2$ term is controlled
by \eqref{eq:local-higher-order-symcurl-Korn} on a fixed enlargement of
$\omega_\chi^{++}$ contained in $\{\chi_0=1\}$, while the last term is
controlled by the scaled Korn estimate modulo the affine kernel that
defines $r_c$.
Together with the transition-region Green bounds, the right-hand side is
$\lesssim d^{-2}$.  Thus each transition annulus contributes at most
$Ch/d$.  Since
$d_j=2^j\kappa_0h$, we conclude that
\begin{equation*}
\begin{aligned}
 \sum_{j=0}^{N}h\|\nabla_h Q_h^{c}\|_{L^1(A_j)}
 &\lesssim
 1+\sum_{\substack{j\ge1:\ A_j^+\subset\{\chi=1\}}}
       \frac{h}{d_j}
 +\frac{h}{d}
 \lesssim1.
\end{aligned}
\end{equation*}

On a nonsource annulus with $A_j^+\subset\{\chi=1\}$, locality
and affine reproduction of $J_h$, with $\symcurl r=0$, allow
$r_e$ to be replaced by an $L^2(A_j^+)$-best representative in
$\mathcal R(A_j^+)$.  By \eqref{eq:HHJ-value-annular-Galerkin-error},
\[
\|\symcurl\varepsilon_h^{(0)}\|_{L^2(A_j^+)}
\lesssim h d_j^{-2}.
\]
The local stability of $J_h$ and the scaled Korn inequality modulo the
affine kernel
then yield
\begin{align*}
\|Q_h^{e}\|_{L^2(A_j)}
&\lesssim
\bigg[
\|\symcurl\varepsilon_h^{(0)}\|_
 {L^2(A_j^+)}
+
d_j^{-1}
\inf_{r\in\mathcal R(A_j^+)}
\|\varepsilon_h^{(0)}-r\|_
 {L^2(A_j^+)}
\bigg]
\\
&\lesssim
\|\symcurl\varepsilon_h^{(0)}\|_
 {L^2(A_j^+)}
\lesssim
hd_j^{-2}.
\end{align*}
Hence $\|Q_h^{e}\|_{L^1(A_j)}
\le
|A_j|^{1/2}
\|Q_h^{e}\|_{L^2(A_j)}
\lesssim
\frac{h}{d_j}$.
On $A_0$, \eqref{eq:HHJ-value-global-Galerkin-energy} and the stability
of $\mathscr C_h$ give $\|Q_h^{e}\|_{L^1(A_0)}\lesssim1$.

On the cutoff transition region, the definition of $r_e$ gives
 \[
 \begin{aligned}
  \|\varepsilon_h^{(0)}-r_e\|_
   {L^2(\omega_\chi^{++})}
  &\le
  \|\Phi^{(0)}-r_c\|_
   {L^2(\omega_\chi^{++})}
  +
  \|\psi_h-r_q\|_{L^2(\omega_\chi^{++})}.
 \end{aligned}
 \]
 The scaled Korn inequality modulo the affine kernel, the continuous Green estimate,
 and the discrete Green transition estimate imply
 \[
 \begin{aligned}
  d^{-1}\|\Phi^{(0)}-r_c\|_
   {L^2(\omega_\chi^{++})}
  &\lesssim \|G^{(0)}\|_
   {L^2(\omega_\chi^{++})}
  \lesssim d^{-1},\\
  d^{-1}\|\psi_h-r_q\|_{L^2(\omega_\chi^{++})}
  &\lesssim \|G_h^{(0)}\|_
   {L^2(\omega_\chi^{++})}
  \lesssim d^{-1}.
 \end{aligned}
 \]
Thus
$\|\varepsilon_h^{(0)}-r_e\|_{L^2(\omega_\chi^{++})}
\lesssim1$.

Local $H^1$ stability, the cutoff product estimate, and
\eqref{eq:HHJ-value-annular-Galerkin-error} at scale $d$ yield
 \[
 \begin{aligned}
  \|Q_h^e\|_{L^2(\omega_\chi^+)}
  &\lesssim \left[
   \|\symcurl\varepsilon_h^{(0)}\|_
    {L^2(\omega_\chi^{++})}
   +d^{-1}
    \|\varepsilon_h^{(0)}-r_e\|_
     {L^2(\omega_\chi^{++})}
  \right]\\
  &\lesssim hd^{-2}+d^{-1}
  \lesssim d^{-1}.
 \end{aligned}
 \]
Since $|\omega_\chi^+|\lesssim d^2$, H\"older's inequality yields
$\|Q_h^e\|_{L^1(\omega_\chi^+)}\lesssim1$.

Only finitely many annuli meet the transition region, all with
$d_j\simeq d$; those outside $\supp Q_h^e$ contribute zero.  Hence
\begin{equation*}
\sum_{j=0}^{N}
 \|Q_h^{e}\|_{L^1(A_j)}
 \lesssim
 1
 + \sum_{\substack{j:\ A_j^+\subset\{\chi=1\}\\j\ge1}}
       \frac{h}{d_j}
 + 1
\lesssim 1.
\end{equation*}

For $p\ge1$, the HHJ interior moment conditions give, on each $T$,
$(\rho_\sigma,q)_T=0$ for all $q\in\mathcal P_0(T;\mathbb S)$.
Because $\mathcal K$ is constant, subtracting the element mean yields
\begin{align*}
\left|
a_T(\rho_\sigma,Q_h^{c})
\right|=
\left|
\bigl(
\rho_\sigma,
\mathcal KQ_h^{c}
-
(\mathcal KQ_h^{c})_T
\bigr)_T
\right|
\lesssim
h
\|\rho_\sigma\|_{L^\infty(T)}
\|\nabla_hQ_h^{c}\|_{L^1(T)}.
\end{align*}
Summing and using the preceding $L^1$ moment bound gives
\[
\left|
a_U(\rho_\sigma,Q_h^{c})
\right|
\lesssim
\|\rho_\sigma\|_{L^\infty(B_{2d}(x_0);\Th)}.
\]
For the Galerkin-error part, the corresponding $L^1$ bound for
$Q_h^e$ implies
\[
\left|
a_U(\rho_\sigma,Q_h^{e})
\right|
\lesssim
\|\rho_\sigma\|_{L^\infty(B_{2d}(x_0);\Th)}.
\]
Combining the two parts of the decomposition proves the log-free value
estimate.
\end{proof}

\begin{theorem}[Localized maximum-norm estimates for HHJ bending moments]
\label{thm:HHJ-local-maximum-norm}
Let $\Sigma_h\times V_h$ be the degree-$p$
HHJ pair on a shape-regular, quasi-uniform triangulation, and assume
that $\mathcal K$ is constant. Assume that the global Dirichlet potential problem
\eqref{eq:HHJ-potential-Dirichlet-dual} satisfies the $H^2$ shift
\eqref{eq:HHJ-potential-H2-shift}; for $p\ge1$, assume in addition
the $H^3$ shift
\eqref{eq:HHJ-potential-H3-shift}.

Let $\mathcal O_{4d}\Subset\Omega$ and
$d\ge\kappa_0 h$, where $\kappa_0$ is sufficiently large but
independent of $h$ and $d$.  For $p=0$,
\begin{equation}
\|\sigma-\sigma_h\|_{L^\infty(\mathcal O;\mathcal T_h)}
\lesssim
\bigl(1+\log(d/h)\bigr)
\|\rho_\sigma\|_{L^\infty(\mathcal O_{3d};\mathcal T_h)}
+
d^{-1}\|\xi_h\|_{L^2(\Omega)}.
\label{eq:HHJ-p0-local-maximum-norm}
\end{equation}

For $p\ge1$,
\begin{equation}
\|\sigma-\sigma_h\|_{L^\infty(\mathcal O;\mathcal T_h)}
\lesssim
\|\rho_\sigma\|_{L^\infty(\mathcal O_{3d};\mathcal T_h)}
+
d^{-2}
\|\xi_h\|_{H^{-1}(\Omega;\mathbb S)}.
\label{eq:HHJ-p-ge-1-local-maximum-norm}
\end{equation}

Moreover, for $p\ge1$,
\begin{align}
&
\|\nabla\sigma-\nabla_h\sigma_h\|_
 {L^\infty(\mathcal O;\mathcal T_h)}
\notag\\
&\qquad\lesssim
\left[
\|\nabla_h\rho_\sigma\|_
 {L^\infty(\mathcal O_{3d};\mathcal T_h)}
+
h^{-1}
\|\rho_\sigma\|_
 {L^\infty(\mathcal O_{3d};\mathcal T_h)}
\right]
+
d^{-3}
\|\xi_h\|_{H^{-1}(\Omega;\mathbb S)}.
\label{eq:HHJ-local-gradient-maximum-norm}
\end{align}
\end{theorem}

\begin{proof}
Fix $x_0\in\mathcal O$ and a unit symmetric tensor $E$.
Using the trace from the selected element $T_0$, value-source
reproduction gives
\begin{equation*}
E:(\sigma-\sigma_h)(x_0)
=
E:\rho_\sigma(x_0)
+
L_0(\xi_h).
\end{equation*}

For $p=0$, Lemma~\ref{lem:HHJ-localized-discrete-Green-identity} gives
\begin{align*}
L_0(\xi_h)
={}&
-a_U\bigl(
\rho_\sigma,
\mathscr C_h(\chi,G_h^{(0)})
\bigr)
+
a_U\bigl(
G_h^{(0)},
\mathscr C_h(\chi,\xi_h)
\bigr)
-
a_U\bigl(
\xi_h,
\mathscr C_h(\chi,G_h^{(0)})
\bigr).
\end{align*}
By symmetry of $a_U$, the two commutator terms pair with the
localization defects.  Their support and $L^2$ bounds in
Lemma~\ref{lem:HHJ-symmetric-potential-commutator}, together with
\eqref{eq:HHJ-value-transition-discrete-Green}, give
\begin{align*}
&
\left|
a_U\bigl(
G_h^{(0)},
\mathscr C_h(\chi,\xi_h)
\bigr)
-
a_U\bigl(
\xi_h,
\mathscr C_h(\chi,G_h^{(0)})
\bigr)
\right|
\lesssim
d^{-1}\|\xi_h\|_{L^2(\Omega)}.
\end{align*}
Combining the localized Green identity, the transition estimate, and
\eqref{eq:HHJ-value-interpolation-Green-pairing} yields
\begin{align*}
|L_0(\xi_h)|
\lesssim
\bigl(1+\log(d/h)\bigr)
\|\rho_\sigma\|_
 {L^\infty(\mathcal O_{3d};\mathcal T_h)}
+
d^{-1}\|\xi_h\|_{L^2(\Omega)}.
\end{align*}
Combining this bound with the preceding pointwise splitting gives
\eqref{eq:HHJ-p0-local-maximum-norm} pointwise.

For $p\ge1$,
\eqref{eq:HHJ-value-localized-Green-bound} and
\eqref{eq:HHJ-value-interpolation-Green-pairing} imply
\begin{equation*}
|L_0(\xi_h)|
\lesssim
\|\rho_\sigma\|_
 {L^\infty(\mathcal O_{3d};\mathcal T_h)}
+
d^{-2}
\|\xi_h\|_{H^{-1}(\Omega;\mathbb S)}.
\end{equation*}
Combining this bound with the preceding pointwise splitting gives
\eqref{eq:HHJ-p-ge-1-local-maximum-norm} pointwise.  Taking
the supremum over $x_0\in\mathcal O$ and the unit symmetric tensors $E$
proves both bounds.

For a unit direction $\nu$, derivative-source reproduction on the selected
element gives
\begin{equation*}
E:\partial_\nu(\sigma-\sigma_h)(x_0)
=
E:\partial_\nu\rho_\sigma(x_0)
+
L_1(\xi_h).
\end{equation*}
By
\eqref{eq:HHJ-derivative-localized-Green-bound} and
\eqref{eq:HHJ-derivative-interpolation-Green-pairing},
\begin{align*}
|L_1(\xi_h)|
\lesssim
h^{-1}
\|\rho_\sigma\|_
 {L^\infty(\mathcal O_{3d};\mathcal T_h)}
+
d^{-3}
\|\xi_h\|_{H^{-1}(\Omega;\mathbb S)}.
\end{align*}
Combining this bound with the preceding derivative splitting and taking
the supremum over $x_0$, $E$, and $\nu$ proves
\eqref{eq:HHJ-local-gradient-maximum-norm}.
\end{proof}

\begin{remark}[Regularity required by the auxiliary potential system]
The $H^3$ shift \eqref{eq:HHJ-potential-H3-shift} is an additional
analytic regularity assumption on the auxiliary potential system.
Convexity of a polygon alone does not imply it.  The conclusions for
$p\ge1$ that invoke this shift are conditional on that assumption.
\end{remark}

\begin{remark}[Extensions to curved and nonconvex domains]
The localization mechanism above is essentially interior.
If $\mathcal O_{4d}\Subset\Omega$, then the local Green problem is
posed on the ball $U=B_{5d/2}(x_0)\Subset\Omega$, and the
symmetric-curl localization operator and localized Green identity do not involve the
physical boundary.

For a smooth curved domain, the present argument suggests an extension
based on the curved-element HHJ formulation and lifting framework of
\cite{ArnoldWalker}.  After lifting the discrete fields to the exact
domain, the additional contributions are geometric terms that must be
estimated together with the discrete Green defect.

For a nonconvex polygon, the interior Green and localization arguments
remain applicable away from reentrant corners, while the global
potential shifts
\eqref{eq:HHJ-potential-H2-shift}--%
\eqref{eq:HHJ-potential-H3-shift}
and the adjoint regularity
\eqref{eq:HHJ-global-adjoint-regularity}
may deteriorate because of corner singularities
\cite{Grisvard1985,Dauge1988}.
Replacing these global shifts by the available fractional or weighted
regularity would lead to corresponding weaker pollution estimates and
convergence rates.  Targets approaching a curved boundary or a
reentrant corner would require boundary or corner Green estimates and
are not considered here.
\end{remark}

\section{A priori estimates and superconvergence}

\subsection{Global negative-norm estimate and a priori rates}
\begin{lemma}[Global negative-norm estimate for the discrete stress error]
\label{lem:HHJ-global-negative-norm}
Let $p\ge1$ and assume $\mathcal K\in W^{1,\infty}(\Omega)$.
For every
$F\in H_0^1(\Omega;\mathbb S)$, suppose that the global HHJ adjoint
problem admits a pair $(\Theta_F,y_F)
 \in
 \bigl(\Sigma\cap H^1(\Omega;\mathbb S)\bigr)
 \times H_0^2(\Omega)$ satisfying
\begin{align*}
 a(\tau,\Theta_F)+b(\tau,y_F)
 &=
 (F,\tau),
 &&\tau\in\Sigma,
 \\
 b(\Theta_F,v)
 &=
 0,
 &&v\in V.
\end{align*}
Assume the adjoint regularity estimate
\cite{AgmonDouglisNirenberg1959,AgmonDouglisNirenberg1964,Grisvard1985,Dauge1988}:
\begin{equation}
 \|\Theta_F\|_{H^1(\Omega)}
 \le
 C\|F\|_{H^1(\Omega)}.
 \label{eq:HHJ-global-adjoint-regularity}
\end{equation}
These adjoint equations control the discrete stress error in $Z_h$
$\xi_h=\Pi_h\sigma-\sigma_h$ through
\begin{equation*}
 \|\xi_h\|_{H^{-1}(\Omega;\mathbb S)}
 \lesssim
 h\left(
   \|\rho_\sigma\|_{L^2(\Omega)}
   +
   \|\xi_h\|_{L^2(\Omega)}
 \right).
\end{equation*}
Combining this preliminary estimate with kernel coercivity gives the
sharper bound
\begin{equation*}
 \|\xi_h\|_{H^{-1}(\Omega;\mathbb S)}
 \lesssim
 h\|\rho_\sigma\|_{L^2(\Omega)}.
\end{equation*}
\end{lemma}

\begin{proof}
Fix $F\in H_0^1(\Omega;\mathbb S)$ and let $(\Theta_F,y_F)$ be the
corresponding global HHJ adjoint pair.
Since $\xi_h\in Z_h$ and $y_F\in H_0^2(\Omega)$, the continuous and
broken pairings agree on $\xi_h$.  The second orthogonality in
\eqref{eq:HHJ-Fortin} implies
\[
 b(\xi_h,y_F)=b_h(\xi_h,y_F)
 =b_h(\xi_h,I_hy_F)=0.
\]
The first adjoint equation gives
\begin{equation*}
 (F,\xi_h)
 =
 a(\xi_h,\Theta_F).
\end{equation*}

By \eqref{eq:HHJ-global-adjoint-regularity}, the $H^1$ traces of
$\Theta_F$ define the normal--normal edge moments of $\Pi_h\Theta_F$; its
interior moments are defined in $L^2$.  For every $v_h\in V_h$, the
first commuting orthogonality in \eqref{eq:HHJ-Fortin} gives
$b_h(\Pi_h\Theta_F,v_h)=b_h(\Theta_F,v_h)$.  Since
$\Theta_F\in H^1(\Omega;\mathbb S)$, the broken and continuous pairings
agree on $V_h$, and hence
\[
 b_h(\Pi_h\Theta_F,v_h)
 =
 b(\Theta_F,v_h)
 =
 0.
\]
Thus $\Pi_h\Theta_F\in Z_h$.
Testing \eqref{eq:kernel-error} with $\tau_h=\Pi_h\Theta_F$ yields
\begin{align*}
 (F,\xi_h)
 &=
 a\bigl(
   \xi_h,\Theta_F-\Pi_h\Theta_F
 \bigr)
 +
 a\bigl(
   \xi_h,\Pi_h\Theta_F
 \bigr)
 \\
 &=
 a\bigl(
   \xi_h,\Theta_F-\Pi_h\Theta_F
 \bigr)
 -
 a\bigl(
   \rho_\sigma,\Pi_h\Theta_F
 \bigr).
\end{align*}

The HHJ interpolation estimate gives
\begin{equation*}
 \left|
 a\bigl(
   \xi_h,\Theta_F-\Pi_h\Theta_F
 \bigr)
 \right|
 \lesssim
 h
 \|\xi_h\|_{L^2(\Omega)}
 \|\Theta_F\|_{H^1(\Omega)}.
\end{equation*}

Since $p\ge1$, the HHJ interior moment conditions give
\begin{equation*}
 (\rho_\sigma,q)_T=0
 \qquad
 \forall q\in\mathcal P_0(T;\mathbb S),
 \quad T\in\mathcal T_h.
\end{equation*}
Let $c_T\in\mathbb S$ be the element mean of $\mathcal K \Theta_F$.  Then
\begin{align*}
 \left|
 (\rho_\sigma,\mathcal K\Pi_h\Theta_F)_T
 \right|
 &=
 \left|
 (\rho_\sigma,
   \mathcal K\Pi_h\Theta_F-c_T)_T
 \right|
 \le
 \|\rho_\sigma\|_{L^2(T)}
 \|\mathcal K\Pi_h\Theta_F-c_T\|_{L^2(T)}.
\end{align*}
The local approximation property, the elementwise Poincar\'e
inequality, and $\mathcal K\in W^{1,\infty}(\Omega)$ imply
\[
 \|\mathcal K\Pi_h\Theta_F-c_T\|_{L^2(T)}
 \lesssim h_T\|\Theta_F\|_{H^1(\omega_T)}.
\]

Summing over the mesh gives
\begin{equation*}
 \left|
 a\bigl(
   \rho_\sigma,\Pi_h\Theta_F
 \bigr)
 \right|
 \lesssim
 h
 \|\rho_\sigma\|_{L^2(\Omega)}
 \|\Theta_F\|_{H^1(\Omega)}.
\end{equation*}

Combining the splitting and the preceding two estimates with
\eqref{eq:HHJ-global-adjoint-regularity} yields
\[
 |(F,\xi_h)|
 \lesssim
 h\left(
   \|\rho_\sigma\|_{L^2(\Omega)}
   +
   \|\xi_h\|_{L^2(\Omega)}
 \right)
 \|F\|_{H^1(\Omega)}.
\]
Taking the supremum over $F\in H_0^1(\Omega;\mathbb S)$ gives the
preliminary estimate.

Coercivity of $a$ on $Z_h$ and \eqref{eq:kernel-error} with
$\tau_h=\xi_h$ yield
\[
 \|\xi_h\|_{L^2(\Omega)}
 \lesssim \|\rho_\sigma\|_{L^2(\Omega)}.
\]
The sharp bound follows from this energy estimate.
\end{proof}

The role of the preceding lemma is solely to close the global pollution
term.  Its adjoint regularity
\eqref{eq:HHJ-global-adjoint-regularity} is
independent of the potential $H^3$ shift
\eqref{eq:HHJ-potential-H3-shift} and is not implied here by convexity
alone; consequently, both the estimates for $p\ge1$ and the optimized
recovery rates below are conditional on it.

\begin{corollary}[Localized maximum-norm estimates in interpolation form]
\label{cor:HHJ-interpolation-form-local-maximum-norm}
Assume the hypotheses of
Theorem~\ref{thm:HHJ-local-maximum-norm}.

For $p=0$,
\begin{equation}
 \|\sigma-\sigma_h\|_{L^\infty(\mathcal O;\mathcal T_h)}
 \lesssim
 \bigl(1+\log(d/h)\bigr)
 \|\rho_\sigma\|_{L^\infty(\mathcal O_{3d};\mathcal T_h)}
 +
 d^{-1}\|\rho_\sigma\|_{L^2(\Omega)}.
 \label{eq:HHJ-p0-interpolation-form-local-maximum-norm}
\end{equation}

Let $p\ge1$.  In addition, assume the global HHJ adjoint regularity
in Lemma~\ref{lem:HHJ-global-negative-norm}.  The value estimate is
\begin{equation*}
 \|\sigma-\sigma_h\|_{L^\infty(\mathcal O;\mathcal T_h)}
 \lesssim
 \|\rho_\sigma\|_{L^\infty(\mathcal O_{3d};\mathcal T_h)}
 +
 hd^{-2}\|\rho_\sigma\|_{L^2(\Omega)}.
\end{equation*}
The corresponding gradient estimate is
\begin{align}
 &
 \|\nabla\sigma-\nabla_h\sigma_h\|_
 {L^\infty(\mathcal O;\mathcal T_h)}
 \notag\\
 &\qquad\lesssim
 \left[
 \|\nabla_h\rho_\sigma\|_
 {L^\infty(\mathcal O_{3d};\mathcal T_h)}
 +
 h^{-1}
 \|\rho_\sigma\|_
 {L^\infty(\mathcal O_{3d};\mathcal T_h)}
 \right]
 +
 hd^{-3}\|\rho_\sigma\|_{L^2(\Omega)}.
 \label{eq:HHJ-interpolation-form-local-gradient-maximum-norm}
\end{align}
\end{corollary}

\begin{proof}
Testing \eqref{eq:kernel-error} with $\xi_h\in Z_h$ and using
ellipticity gives
$\|\xi_h\|_{L^2(\Omega)}\lesssim\|\rho_\sigma\|_{L^2(\Omega)}$.
Substitution in \eqref{eq:HHJ-p0-local-maximum-norm} proves
\eqref{eq:HHJ-p0-interpolation-form-local-maximum-norm}.  For $p\ge1$, substitute
Lemma~\ref{lem:HHJ-global-negative-norm} into
\eqref{eq:HHJ-p-ge-1-local-maximum-norm} and
\eqref{eq:HHJ-local-gradient-maximum-norm} to obtain the claimed
value and gradient estimates in interpolation form.
\end{proof}

\begin{corollary}[Interior a priori rates for HHJ bending moments]
\label{cor:HHJ-interior-apriori-rates}
Under the hypotheses of
Corollary~\ref{cor:HHJ-interpolation-form-local-maximum-norm}, assume that $ \sigma
 \in
 W^{p+1,\infty}(\mathcal O_{4d};\mathbb S)
 \cap
 H^{p+1}(\Omega;\mathbb S)$.

For $p=0$,
\begin{align*}
 \|\sigma-\sigma_h\|_{L^\infty(\mathcal O;\mathcal T_h)}
 &\lesssim
 h\bigl(1+\log(d/h)\bigr)
 |\sigma|_{W^{1,\infty}(\mathcal O_{4d})}
 +
 hd^{-1}|\sigma|_{H^1(\Omega)}.
\end{align*}

For $p\ge1$,
\begin{align*}
 \|\sigma-\sigma_h\|_{L^\infty(\mathcal O;\mathcal T_h)}
 &\lesssim
 h^{p+1}
 |\sigma|_{W^{p+1,\infty}(\mathcal O_{4d})}
 +
 h^{p+2}d^{-2}
 |\sigma|_{H^{p+1}(\Omega)},
 \\
 \|\nabla\sigma-\nabla_h\sigma_h\|_
 {L^\infty(\mathcal O;\mathcal T_h)}
 &\lesssim
 h^p
 |\sigma|_{W^{p+1,\infty}(\mathcal O_{4d})}
 +
 h^{p+2}d^{-3}
 |\sigma|_{H^{p+1}(\Omega)}.
\end{align*}

In particular, if $d>0$ is fixed independently of $h$, then
\[
 \|\sigma-\sigma_h\|_{L^\infty(\mathcal O;\mathcal T_h)}
 =
 \begin{cases}
  O\!\left(h(1+|\log h|)\right),&p=0,\\
  O(h^{p+1}),&p\ge1
 \end{cases}.
\]
The gradient rate is
\[
 \|\nabla\sigma-\nabla_h\sigma_h\|_
 {L^\infty(\mathcal O;\mathcal T_h)}
 =
 O(h^p),
 \qquad p\ge1.
\]
\end{corollary}

\begin{proof}
The canonical HHJ interpolation estimates give
\begin{align*}
 \|\rho_\sigma\|_{L^\infty(\mathcal O_{3d};\mathcal T_h)}
 &\lesssim
 h^{p+1}
 |\sigma|_{W^{p+1,\infty}(\mathcal O_{4d})},\\
 \|\nabla_h\rho_\sigma\|_{L^\infty(\mathcal O_{3d};\mathcal T_h)}
 &\lesssim
 h^p
 |\sigma|_{W^{p+1,\infty}(\mathcal O_{4d})},\\
 \|\rho_\sigma\|_{L^2(\Omega)}
 &\lesssim
 h^{p+1}|\sigma|_{H^{p+1}(\Omega)}.
\end{align*}
Substitution into
\eqref{eq:HHJ-p0-interpolation-form-local-maximum-norm}--
\eqref{eq:HHJ-interpolation-form-local-gradient-maximum-norm}
gives the stated estimates.  The fixed-region rates follow by taking
$d$ independent of $h$.
\end{proof}

\subsection{Symmetric recovery and parity}

Fix an interior mesh vertex $x_0$, let $\mathcal S_h(x_0)$ denote its
element star, and define the antipodal reflection
$\mathscr Sx:=2x_0-x$.  The recovery construction below pairs elements
in this star through $\mathscr S$, so its gain is governed by the parity
of the interpolation error under this reflection.

\begin{assumption}[Local reflection symmetry]
\label{ass:HHJ-local-reflection-symmetry}
Assume $B_{4d}(x_0)\Subset\Omega$ and the following conditions.
\begin{enumerate}
\item The local mesh patches in $B_{4d}(x_0)$ used in the pointwise
construction are invariant under $\mathscr S$, with elements paired
as $T$ and $\mathscr S T$; in particular, $\mathcal S_h(x_0)$ is invariant.
The local HHJ stress and potential spaces are invariant under the
reflection pullbacks, and the canonical interpolant satisfies
$\Pi_h(q\circ\mathscr S)=(\Pi_hq)\circ\mathscr S$ on these pairs.
The regularized value and derivative sources on $T$ and $\mathscr S T$
are chosen equivariantly, and $\mathcal K$ is reflection invariant in
Theorem~\ref{thm:HHJ-local-maximum-norm}.
\item The recovery weights satisfy
\[
 \lambda_T=\lambda_{\mathscr S T},
 \qquad
 \sum_{T\in\mathcal S_h(x_0)}\lambda_T=1,
 \qquad
 \sum_{T\in\mathcal S_h(x_0)}|\lambda_T|\lesssim1.
\]
\end{enumerate}
\end{assumption}

For vector potentials, set $\mathscr P v:=-v\circ\mathscr S$.  Then
$\symcurl(\mathscr P v)=(\symcurl v)\circ\mathscr S$.  For the
antipodal reflection, the natural pullback of a symmetric tensor
reduces to composition, since
$(-I)\tau(-I)^T=\tau$.

\begin{assumption}[Reflection-compatible discrete Green construction]
\label{ass:HHJ-Galerkin-reflection}
The global potential space is invariant under $\mathscr P$, and the
global Galerkin projection in \eqref{eq:HHJ-potential-Galerkin-Green} commutes with
$\mathscr P$: whenever
\[
 \mathcal A_\Omega(\Phi_h,\psi_h)
 =
 \mathcal A_\Omega(v,\psi_h)
 \qquad
 \forall\psi_h\in W_{h,0},
\]
the reflected fields satisfy
\[
 \mathcal A_\Omega(\mathscr P\Phi_h,\psi_h)
 =
 \mathcal A_\Omega(\mathscr Pv,\psi_h)
 \qquad
 \forall\psi_h\in W_{h,0}.
\]
All auxiliary fields, cutoffs, and operators in the paired Green-stress
localization construction are equivariant under the reflection
pullbacks, including the local action of $J_h$, the representatives
$\Phi^{(m)}$, $r_0^{(m)}$ in
$v^{(m)}=\chi_0(\Phi^{(m)}-r_0^{(m)})$, and the affine representatives
entering $\mathscr C_h$.
\end{assumption}

\begin{remark}
Assumption~\ref{ass:HHJ-Galerkin-reflection} follows from uniqueness of the
Galerkin problem if $\Omega$, the global triangulation, $\mathcal K$, the
homogeneous boundary conditions, and the cutoff construction are all
invariant under $\mathscr S$.  Local symmetry of the star alone does
not imply global Galerkin compatibility. 
\end{remark}

Define the recovered value and recovered broken gradient by
\begin{align*}
 \mathscr M_h\sigma_h(x_0)
 &:=
 \sum_{T\in\mathcal S_h(x_0)}
 \lambda_T\,\sigma_h|_T(x_0),
 \\
 \mathscr M_h^\nabla\sigma_h(x_0)
 &:=
 \sum_{T\in\mathcal S_h(x_0)}
 \lambda_T\,\nabla(\sigma_h|_T)(x_0).
\end{align*}
For a tensor field $q$, write
\[
 q^+(x):=\frac12\bigl(q(x)+q(\mathscr Sx)\bigr),
 \qquad
 q^-(x):=\frac12\bigl(q(x)-q(\mathscr Sx)\bigr).
\]
In particular,
\[
 \sigma(x_0)=\sigma^+(x_0),
 \qquad
 \nabla\sigma(x_0)=\nabla\sigma^-(x_0).
\]
Recovery therefore selects even values and odd first derivatives.

\begin{lemma}[Parity-improved HHJ interpolation]
\label{lem:HHJ-parity-interpolation}
Assume the reflection equivariance in
Assumption~\ref{ass:HHJ-local-reflection-symmetry} and let
$\rho_\sigma=\sigma-\Pi_h\sigma$.

If $p\ge0$ is even and
$\sigma\in W^{p+2,\infty}(B_{4d}(x_0);\mathbb S)$, then
$(\rho_\sigma)^+=\sigma^+-\Pi_h\sigma^+$, and
\begin{equation*}
 \|(\rho_\sigma)^+\|_
 {L^\infty(B_{3d}(x_0);\mathcal T_h)}
 \lesssim
 h^{p+1}d
 |\sigma|_{W^{p+2,\infty}(B_{4d}(x_0))}.
\end{equation*}

If $p\ge1$ is odd and
$\sigma\in W^{p+2,\infty}(B_{4d}(x_0);\mathbb S)$, then
$(\rho_\sigma)^-=\sigma^--\Pi_h\sigma^-$, and
\begin{align*}
 \|\nabla_h(\rho_\sigma)^-\|_
 {L^\infty(B_{3d}(x_0);\mathcal T_h)}
 +
 h^{-1}\|(\rho_\sigma)^-\|_
 {L^\infty(B_{3d}(x_0);\mathcal T_h)}
\lesssim
 h^p d
 |\sigma|_{W^{p+2,\infty}(B_{4d}(x_0))}.
\end{align*}
\end{lemma}

\begin{proof}
For even $p$, the derivatives of odd total order of $\sigma^+$ vanish
at $x_0$.  Hence
\[
 D^\beta\sigma^+(x_0)=0,
 \qquad |\beta|=p+1,
\]
and Taylor's theorem yields
\[
 |\sigma^+|_{W^{p+1,\infty}(B_{3d}(x_0))}
 \lesssim
 d|\sigma|_{W^{p+2,\infty}(B_{4d}(x_0))}.
\]
The HHJ interpolation estimate proves the even-degree bound.

For odd $p$, every derivative of even total order of $\sigma^-$
vanishes at $x_0$.  Since $p+1$ is even,
\[
 D^\beta\sigma^-(x_0)=0,
 \qquad |\beta|=p+1.
\]
Taylor's theorem and the value and gradient interpolation estimates
prove the odd-degree bound.  The identities for
$(\rho_\sigma)^\pm$ follow from interpolation equivariance.
\end{proof}

\subsection{Recovered values, derivatives, and optimized rates}

For $T\in\mathcal S_h(x_0)$ and $m\in\{0,1\}$, let
$G_{h,T}^{(m)}$ denote the discrete Green stress constructed with the
corresponding regularized source supported in $T$.

\begin{theorem}[Recovered bending-moment values]
Let $p\ge2$ be even.  Assume the hypotheses of
Theorem~\ref{thm:HHJ-local-maximum-norm},
Assumptions~\ref{ass:HHJ-local-reflection-symmetry}
and~\ref{ass:HHJ-Galerkin-reflection}, and suppose that
\[
 B_{4d}(x_0)\Subset\Omega,
 \qquad
 d\ge\kappa_0h,
 \qquad
 \sigma\in W^{p+2,\infty}(B_{4d}(x_0);\mathbb S).
\]
The recovered value satisfies
\begin{align}
 |\sigma(x_0)-\mathscr M_h\sigma_h(x_0)|
 \lesssim{}&
 h^{p+1}d
 |\sigma|_{W^{p+2,\infty}(B_{4d}(x_0))}
 +
 d^{-2}\|\xi_h\|_{H^{-1}(\Omega;\mathbb S)}.
 \label{eq:HHJ-recovered-value}
\end{align}
\end{theorem}

\begin{proof}
Fix a tensor component $E$.  For each $T\in\mathcal S_h(x_0)$, use
the element-trace value representation from
Theorem~\ref{thm:HHJ-local-maximum-norm}, with source supported in $T$.
By \eqref{eq:HHJ-value-localized-Green-bound} and
\eqref{eq:HHJ-value-interpolation-Green-pairing}, the remainder is bounded
by $d^{-2}\|\xi_h\|_{H^{-1}(\Omega;\mathbb S)}$.

Pair $T$ with $\mathscr S T$ and set
\[
 Q_{h,T}^{(0)}
 :=
 \mathscr C_h\bigl(\chi,G_{h,T}^{(0)}\bigr).
\]
The two reflection assumptions imply
\[
 Q_{h,\mathscr S T}^{(0)}(x)
 =
 Q_{h,T}^{(0)}(\mathscr Sx).
\]
Since $U$ is invariant under $\mathscr S$,
$|\det D\mathscr S|=1$, and $\mathcal K$ is reflection invariant, a
change of variables yields
\begin{align*}
 &a_U\bigl(\rho_\sigma,Q_{h,T}^{(0)}\bigr)
 +a_U\bigl(\rho_\sigma,Q_{h,\mathscr S T}^{(0)}\bigr)
 \\
 &\qquad=
 a_U\bigl(\rho_\sigma,Q_{h,T}^{(0)}\bigr)
 +a_U\bigl(\rho_\sigma\circ\mathscr S,
            Q_{h,T}^{(0)}\bigr)
 \\
 &\qquad=
 2a_U\bigl((\rho_\sigma)^+,Q_{h,T}^{(0)}\bigr).
\end{align*}
Equal reflected weights, their stability, and the localized Green
bound give
\[
 |\sigma(x_0)-\mathscr M_h\sigma_h(x_0)|
 \lesssim
 \|(\rho_\sigma)^+\|_
 {L^\infty(B_{3d}(x_0);\mathcal T_h)}
 +
 d^{-2}\|\xi_h\|_{H^{-1}(\Omega;\mathbb S)}.
\]
Lemma~\ref{lem:HHJ-parity-interpolation} gives \eqref{eq:HHJ-recovered-value}.
\end{proof}

\begin{remark}[The lowest-order case]
\label{rem:HHJ-lowest-order-recovery}
For $p=0$, reflection improves local interpolation, but the pollution
term $d^{-1}\|\xi_h\|_{L^2(\Omega)}$ remains.  The present argument
therefore gives no fully a priori lowest-order superconvergence rate.
\end{remark}

\begin{theorem}[Recovered first derivatives]
Let $p\ge1$ be odd.  Assume the hypotheses of
Theorem~\ref{thm:HHJ-local-maximum-norm},
Assumptions~\ref{ass:HHJ-local-reflection-symmetry}
and~\ref{ass:HHJ-Galerkin-reflection}, and suppose that
\[
 B_{4d}(x_0)\Subset\Omega,
 \qquad
 d\ge\kappa_0h,
 \qquad
 \sigma\in W^{p+2,\infty}(B_{4d}(x_0);\mathbb S).
\]
The recovered gradient satisfies
\begin{align}
 &
 |\nabla\sigma(x_0)-\mathscr M_h^\nabla\sigma_h(x_0)|
\lesssim
 h^p d
 |\sigma|_{W^{p+2,\infty}(B_{4d}(x_0))}
 +
 d^{-3}\|\xi_h\|_{H^{-1}(\Omega;\mathbb S)}.
 \label{eq:HHJ-recovered-derivative}
\end{align}
\end{theorem}

\begin{proof}
Apply the element-trace derivative representation in the proof of
Theorem~\ref{thm:HHJ-local-maximum-norm} on every
$T\in\mathcal S_h(x_0)$, and set
\[
 Q_{h,T}^{(1)}
 :=\mathscr C_h\bigl(\chi,G_{h,T}^{(1)}\bigr).
\]
The derivative source is odd under $\mathscr S$, and the two reflection
assumptions imply
\[
 Q_{h,\mathscr S T}^{(1)}(x)
 =
 -Q_{h,T}^{(1)}(\mathscr Sx).
\]
Reflection invariance of $U$ and $\mathcal K$ and the identity
$|\det D\mathscr S|=1$ yield
\begin{align*}
 &a_U\bigl(\rho_\sigma,Q_{h,T}^{(1)}\bigr)
 +a_U\bigl(\rho_\sigma,Q_{h,\mathscr S T}^{(1)}\bigr)
 \\
 &\qquad=
 a_U\bigl(\rho_\sigma,Q_{h,T}^{(1)}\bigr)
 -a_U\bigl(\rho_\sigma\circ\mathscr S,
            Q_{h,T}^{(1)}\bigr)
 \\
 &\qquad=
 2a_U\bigl((\rho_\sigma)^-,Q_{h,T}^{(1)}\bigr).
\end{align*}
As in the value case, equal reflected weights, their stability,
\eqref{eq:HHJ-derivative-localized-Green-bound}, and
\eqref{eq:HHJ-derivative-interpolation-Green-pairing} give
\begin{align*}
 &
 |\nabla\sigma(x_0)-\mathscr M_h^\nabla\sigma_h(x_0)|
 \\
 &\qquad\lesssim
 \|\nabla_h(\rho_\sigma)^-\|_
 {L^\infty(B_{3d}(x_0);\mathcal T_h)}
 +
 h^{-1}\|(\rho_\sigma)^-\|_
 {L^\infty(B_{3d}(x_0);\mathcal T_h)}
 +
 d^{-3}\|\xi_h\|_{H^{-1}(\Omega;\mathbb S)}.
\end{align*}
Lemma~\ref{lem:HHJ-parity-interpolation} gives \eqref{eq:HHJ-recovered-derivative}.
\end{proof}

\begin{corollary}[Optimized recovery estimates]
\label{cor:HHJ-optimized-superconvergence}
Assume the hypotheses of the corresponding recovery theorem and the
global HHJ adjoint regularity in
Lemma~\ref{lem:HHJ-global-negative-norm}, and suppose in addition that
$\sigma\in H^{p+1}(\Omega;\mathbb S)$.

For even $p\ge2$,
\begin{align*}
 |\sigma(x_0)-\mathscr M_h\sigma_h(x_0)|
 \lesssim{}&
 h^{p+1}d
 |\sigma|_{W^{p+2,\infty}(B_{4d}(x_0))}
 +
 h^{p+2}d^{-2}|\sigma|_{H^{p+1}(\Omega)}.
\end{align*}
For odd $p\ge1$,
\begin{align*}
 &
 |\nabla\sigma(x_0)-\mathscr M_h^\nabla\sigma_h(x_0)|
\lesssim
 h^p d
 |\sigma|_{W^{p+2,\infty}(B_{4d}(x_0))}
 +
 h^{p+2}d^{-3}|\sigma|_{H^{p+1}(\Omega)}.
\end{align*}

If $x_0$ is fixed independently of $h$, then choosing $d=h^{1/3}$
for even $p\ge2$ yields
\begin{equation*}
 |\sigma(x_0)-\mathscr M_h\sigma_h(x_0)|
 \lesssim
 h^{p+4/3}
 \left(
 |\sigma|_{W^{p+2,\infty}(B_{4d}(x_0))}
 +
 |\sigma|_{H^{p+1}(\Omega)}
 \right),
\end{equation*}
whereas choosing $d=h^{1/2}$ for odd $p\ge1$ yields
\begin{align*}
 &
 |\nabla\sigma(x_0)-\mathscr M_h^\nabla\sigma_h(x_0)|
\lesssim
 h^{p+1/2}
 \left(
 |\sigma|_{W^{p+2,\infty}(B_{4d}(x_0))}
 +
 |\sigma|_{H^{p+1}(\Omega)}
 \right).
\end{align*}
\end{corollary}

\begin{proof}
Lemma~\ref{lem:HHJ-global-negative-norm} and the global HHJ
interpolation estimate yield
\[
 \|\xi_h\|_{H^{-1}(\Omega;\mathbb S)}
 \lesssim
 h\|\rho_\sigma\|_{L^2(\Omega)}
 \lesssim
 h^{p+2}|\sigma|_{H^{p+1}(\Omega)}.
\]
Substitution into \eqref{eq:HHJ-recovered-value} and
\eqref{eq:HHJ-recovered-derivative} gives the stated estimates.

For $d=h^\alpha$, the two powers in the even case are
$p+1+\alpha$ and $p+2-2\alpha$, which balance at
$\alpha=1/3$.  In the odd derivative case they are
$p+\alpha$ and $p+2-3\alpha$, which balance at $\alpha=1/2$.
Both choices satisfy $d/h\to\infty$, hence $d\ge\kappa_0h$ for
sufficiently small $h$; for fixed interior $x_0$,
$B_{4d}(x_0)\Subset\Omega$ also holds for sufficiently small $h$.
\end{proof}

\begin{remark}[Proved recovery orders]
The theory proves $h^{p+4/3}$ for even-degree recovered values and
$h^{p+1/2}$ for odd-degree recovered first derivatives.  The full orders
$h^{p+2}$ and $h^{p+1}$ observed on globally uniform symmetric meshes
require additional cancellation of the pollution term, which is not
established here.
\end{remark}

\begin{remark}[Lowest-order mixed and nonconforming correspondences]
The equivalence between the lowest-order HHJ method and the modified
Morley method identifies the HHJ moment with the corresponding corrected
broken Hessian of the Morley solution
\cite{ArnoldBrezzi1985,HuMa2016}.  Consequently, the localized
maximum-norm estimate
\eqref{eq:HHJ-p0-interpolation-form-local-maximum-norm} and the
conditional symmetry improvement in
Remark~\ref{rem:HHJ-lowest-order-recovery} transfer to the corrected
broken Hessian of the Morley solution.  The present argument does not,
however, yield a fully a priori lowest-order Morley superconvergence
rate because the pollution term remains.  Analogous second-order
equivalences on triangular and rectangular meshes carry
Raviart--Thomas superconvergence to the corresponding nonconforming
methods
\cite{HuMa2016,Li2018,LiRannacherTurek2021}.
\end{remark}

\section{Numerical experiments}
All numerical experiments were performed with NGSolve
\cite{Schoeberl2014}.  We test the predicted parity dependence of
symmetric recovery by comparing bending-moment values and elementwise
first derivatives for even and odd stress degrees on uniform and
nonuniform meshes.  No localization scale $d$ enters the computation;
the choices of $d$ in
Corollary~\ref{cor:HHJ-optimized-superconvergence} arise only in the
analytical optimization of the estimates.

\subsection{Test problem, meshes, and error measures}

Let
\[
 \Omega=(-1,1)^2,
 \qquad
 x_0=(0,0).
\]
We take
$\mathcal C=\mathcal K=\operatorname{Id}_{\mathbb S}$, the identity on
$\mathbb S$, and use the smooth manufactured
deflection
\[
 u(x,y)=(1-x^2)^2(1-y^2)^2e^{x+2y}.
\]
Then $\sigma=\nabla^2u$ and $f=\DivDiv\sigma=\Delta^2u$, while the
squared boundary factors enforce the clamped conditions
$u=\partial_n u=0$ on $\partial\Omega$.

Joining $x_0$ to the four corners $(\pm1,\pm1)$ and four side
midpoints $(\pm1,0)$ and $(0,\pm1)$ produces an initial mesh
$\mathcal T_0$ of eight triangles.  Each subsequent mesh
$\mathcal T_\ell$ is obtained by $\ell$ uniform refinements.  Its mesh
size is
\[
 h_\ell:=\max_{T\in\mathcal T_\ell}\operatorname{diam}(T)
 =\sqrt2\,2^{-\ell}.
\]
The reflection $\mathscr Sx=-x$ leaves every $\mathcal T_\ell$
invariant, and the central star contains eight elements in four
reflection pairs.  Equal recovery weights $1/8$, together with the
reflection invariance of the domain, coefficient, boundary conditions,
and discrete constructions, satisfy
Assumptions~\ref{ass:HHJ-local-reflection-symmetry}
and~\ref{ass:HHJ-Galerkin-reflection}.  We use $p=2$ to test recovered
values and $p=3$ to test recovered first derivatives; the $p=2$
derivative and $p=3$ value serve as comparisons for the opposite parity.

The selected element on level $\ell$ is fixed geometrically as
\[
 T_{0,\ell}
 :=\operatorname{conv}\bigl\{(-2^{-\ell},0),
 (-2^{-\ell},-2^{-\ell}),(0,0)\bigr\}.
\]
Because an HHJ stress need not be single-valued at a vertex, the value
and gradient at $x_0$ are evaluated from
$\sigma_h|_{T_{0,\ell}}$.  Figure~\ref{fig:HHJ-mesh-structure} shows
this element, its reflected partner, and the central recovery patch.

\begin{figure}[!htbp]
\centering
\includegraphics[width=0.96\textwidth]{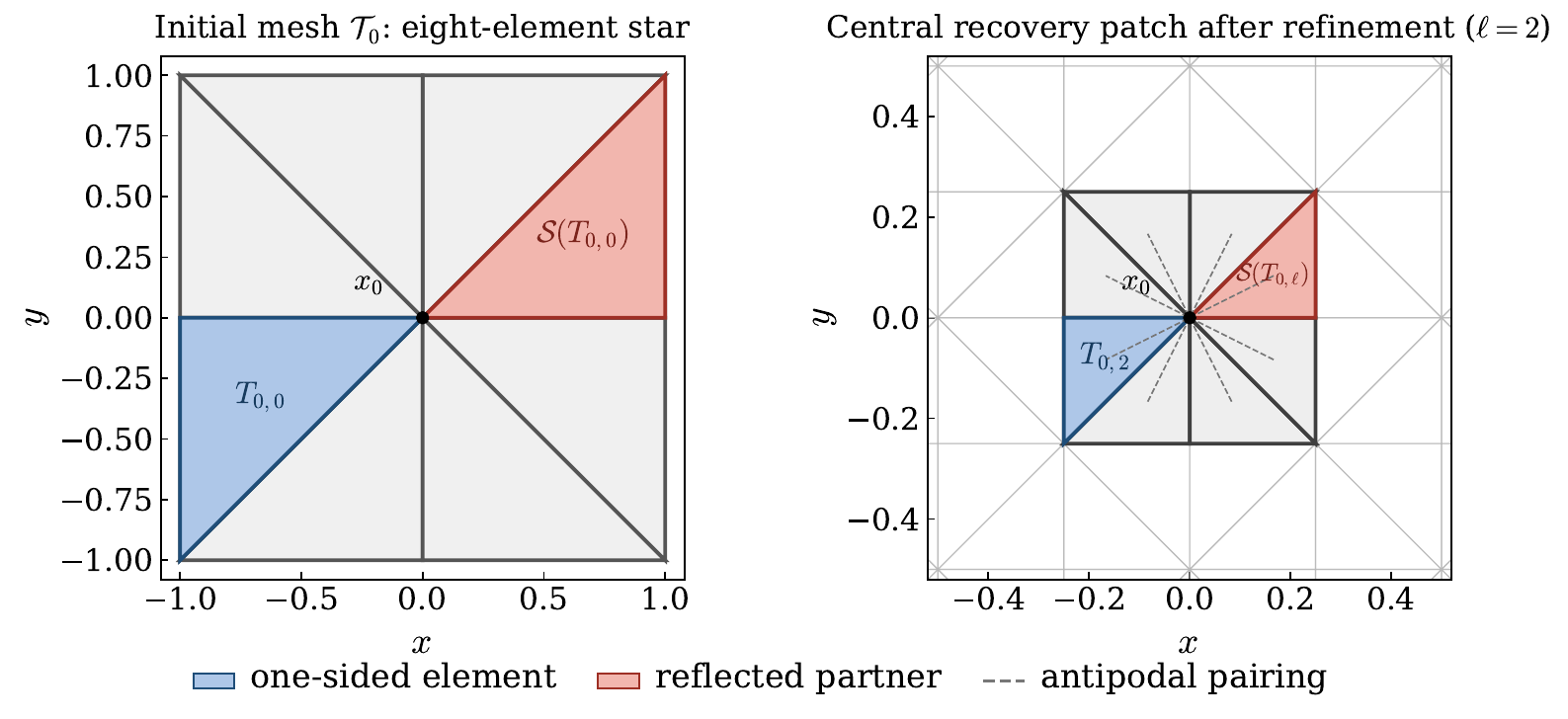}
\caption{Initial mesh and central recovery patch after two uniform
refinements. The highlighted triangles form a reflection pair about
$x_0$.}
\label{fig:HHJ-mesh-structure}
\end{figure}

At each level, the selected-element and recovered quantities use the same HHJ
stress approximation, so recovery is purely a postprocessing step.  All
elementwise values and derivatives are expressed in the same Cartesian
frame before averaging.  Here $|\cdot|_F$ denotes the Frobenius norm
for matrices and its Euclidean extension to tensor arrays.  The four
errors are
\begin{align*}
 e_{\mathrm{elem},\ell}^{(0)}
 &:={}
 \left|
  \sigma(x_0)-\sigma_h|_{T_{0,\ell}}(x_0)
 \right|_F,
 &
 e_{\mathrm{rec},\ell}^{(0)}
 &:={}
 \left|
  \sigma(x_0)-\mathscr M_h\sigma_h(x_0)
 \right|_F,
 \\
 e_{\mathrm{elem},\ell}^{(1)}
 &:={}
 \left|
  \nabla\sigma(x_0)
  -\nabla(\sigma_h|_{T_{0,\ell}})(x_0)
 \right|_F,
 &
 e_{\mathrm{rec},\ell}^{(1)}
 &:={}
 \left|
  \nabla\sigma(x_0)
  -\mathscr M_h^\nabla\sigma_h(x_0)
 \right|_F.
\end{align*}
For each error sequence $e_\ell$, the adjacent-level observed order is
defined by
\begin{equation*}
 \operatorname{rate}_\ell
 =
 \frac{\log(e_{\ell-1}/e_\ell)}
 {\log(h_{\ell-1}/h_\ell)},
 \qquad \ell\ge1.
\end{equation*}
Thus the rate in row $\ell\geq1$ of
Tables~\ref{tab:HHJ-p2-errors} and~\ref{tab:HHJ-p3-errors} uses levels
$\ell-1$ and $\ell$.

\begin{table}[!htbp]
\centering
\small
\setlength{\tabcolsep}{2pt}
\renewcommand{\arraystretch}{1.05}
\caption{Degree-$2$ pointwise errors; the value sequence includes two
additional refinement levels and $h_\ell=\sqrt2\,2^{-\ell}$.}
\label{tab:HHJ-p2-errors}
\begin{tabular}{@{}rrrrr@{\qquad}rrrr@{}}
\hline
& \multicolumn{4}{c}{\textit{(A) Bending-moment value}}
& \multicolumn{4}{c}{\textit{(B) Bending-moment gradient}}\\[2pt]
$\ell$
& $e_{\mathrm{elem}}^{(0)}$ & rate
& $e_{\mathrm{rec}}^{(0)}$ & rate
& $e_{\mathrm{elem}}^{(1)}$ & rate
& $e_{\mathrm{rec}}^{(1)}$ & rate\\
\hline
0 & 2.951e0 & --   & 1.062e1 & --
  & 1.686e1 & --   & 7.637e1 & --\\
1 & 1.061e-1 & 4.80 & 7.008e-1 & 3.92
  & 4.385e0 & 1.94 & 6.273e0 & 3.61\\
2 & 2.294e-2 & 2.21 & 3.953e-2 & 4.15
  & 9.094e-1 & 2.27 & 1.148e0 & 2.45\\
3 & 5.352e-3 & 2.10 & 2.400e-3 & 4.04
  & 2.588e-1 & 1.81 & 2.852e-1 & 2.01\\
4 & 8.737e-4 & 2.62 & 1.488e-4 & 4.01
  & 7.797e-2 & 1.73 & 7.172e-2 & 1.99\\
5 & 1.235e-4 & 2.82 & 9.281e-6 & 4.00
  & -- & -- & -- & --\\
6 & 1.638e-5 & 2.92 & 5.794e-7 & 4.00
  & -- & -- & -- & --\\
\hline
\end{tabular}
\end{table}

\begin{table}[!htbp]
\centering
\small
\setlength{\tabcolsep}{2pt}
\renewcommand{\arraystretch}{1.05}
\caption{Degree-$3$ pointwise errors, with
$h_\ell=\sqrt2\,2^{-\ell}$.}
\label{tab:HHJ-p3-errors}
\begin{tabular}{@{}rrrrr@{\qquad}rrrr@{}}
\hline
& \multicolumn{4}{c}{\textit{(A) Bending-moment value}}
& \multicolumn{4}{c}{\textit{(B) Bending-moment gradient}}\\[2pt]
$\ell$
& $e_{\mathrm{elem}}^{(0)}$ & rate
& $e_{\mathrm{rec}}^{(0)}$ & rate
& $e_{\mathrm{elem}}^{(1)}$ & rate
& $e_{\mathrm{rec}}^{(1)}$ & rate\\
\hline
0 & 7.298e-1 & --   & 2.125e0 & --
  & 7.452e0 & --   & 5.429e1 & --\\
1 & 7.904e-2 & 3.21 & 1.239e-1 & 4.10
  & 1.434e0 & 2.38 & 2.929e0 & 4.21\\
2 & 7.629e-3 & 3.37 & 7.212e-3 & 4.10
  & 3.507e-1 & 2.03 & 1.988e-1 & 3.88\\
3 & 6.000e-4 & 3.67 & 4.317e-4 & 4.06
  & 5.817e-2 & 2.59 & 1.269e-2 & 3.97\\
4 & 4.174e-5 & 3.85 & 2.663e-5 & 4.02
  & 8.206e-3 & 2.83 & 7.970e-4 & 3.99\\
\hline
\end{tabular}
\end{table}

Table~\ref{tab:HHJ-p2-errors} shows that the selected-element bending-moment
value enters its asymptotic regime only on the finer meshes, where its
rate approaches three, while the recovered value stabilizes near fourth
order.  Both $p=2$ gradient sequences remain on the ordinary
second-order scale.  Table~\ref{tab:HHJ-p3-errors} shows the complementary
behavior: both values approach fourth order, whereas the selected-element and
recovered gradients approach third and fourth order, respectively.
Thus the recovery gain switches from the value for even $p=2$ to the
first derivative for odd $p=3$, as predicted by the parity mechanism
under reflection.  On these uniform meshes, each quantity predicted to
superconverge gains about
one order, exceeding the improvement proved in
Corollary~\ref{cor:HHJ-optimized-superconvergence}.

\subsection{Nonuniform mesh comparisons}

For each $\ell$, Families~II and III retain the connectivity of
$\mathcal T_{\ell+1}$ and perturb its vertices.  Thus the family index
$\ell$ refers to the perturbed mesh constructed from uniform level
$\ell+1$.  In Family~II, each interior pair $\{x,-x\}$ is moved
according to
\begin{align*}
 x^*&=x+r_{\ell,x}(\cos\theta_{\ell,x},\sin\theta_{\ell,x}),
 &(-x)^*&=-x^*,\\
 \theta_{\ell,x}&\sim\operatorname{Unif}[0,2\pi),
 &r_{\ell,x}&\sim\operatorname{Unif}
 [0.054h_{\ell+1},0.12h_{\ell+1}].
\end{align*}
Boundary vertices and $x_0$ remain fixed, and the random seed is
$20260904+1009\ell$.

Family~III additionally moves the three central-ring vertices nearest
the directions $-\pi$, $-\pi/2$, and $\pi/4$ by $0.15h_{\ell+1}$ in
the respective normalized directions $(0.82,0.57)$,
$(-0.64,0.77)$, and $(0.73,-0.68)$, while leaving their reflected
partners fixed.  Thus Family~II preserves reflection symmetry, whereas
Family~III breaks it locally at $x_0$, as shown in
Figure~\ref{fig:HHJ-three-mesh-families}.

\begin{figure}[!htbp]
\centering
\includegraphics[width=0.98\textwidth]
{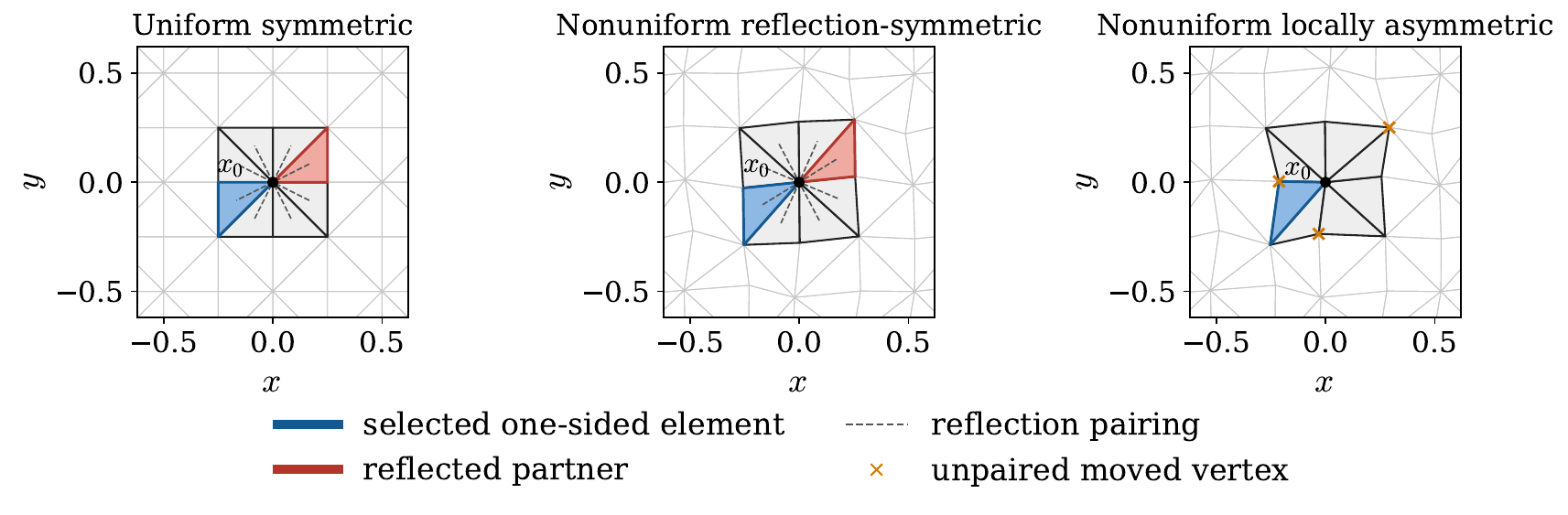}
\caption{Uniform symmetric, nonuniform reflection-symmetric, and locally
asymmetric mesh families. The blue triangle is the selected element.}
\label{fig:HHJ-three-mesh-families}
\end{figure}

The reported meshes remain shape regular.  The minimum angles of
Families~II and III are $25.3$ and $19.8$ degrees, respectively, and
each central star has eight elements.  For Family~II, the largest
observed ratio of the largest to the smallest element diameter is
$1.585$.  Figure~\ref{fig:HHJ-recovery-symmetry-comparison}
compares the recovered errors across all three families.

\begin{figure}[!htbp]
\centering
\includegraphics[width=0.94\textwidth]
{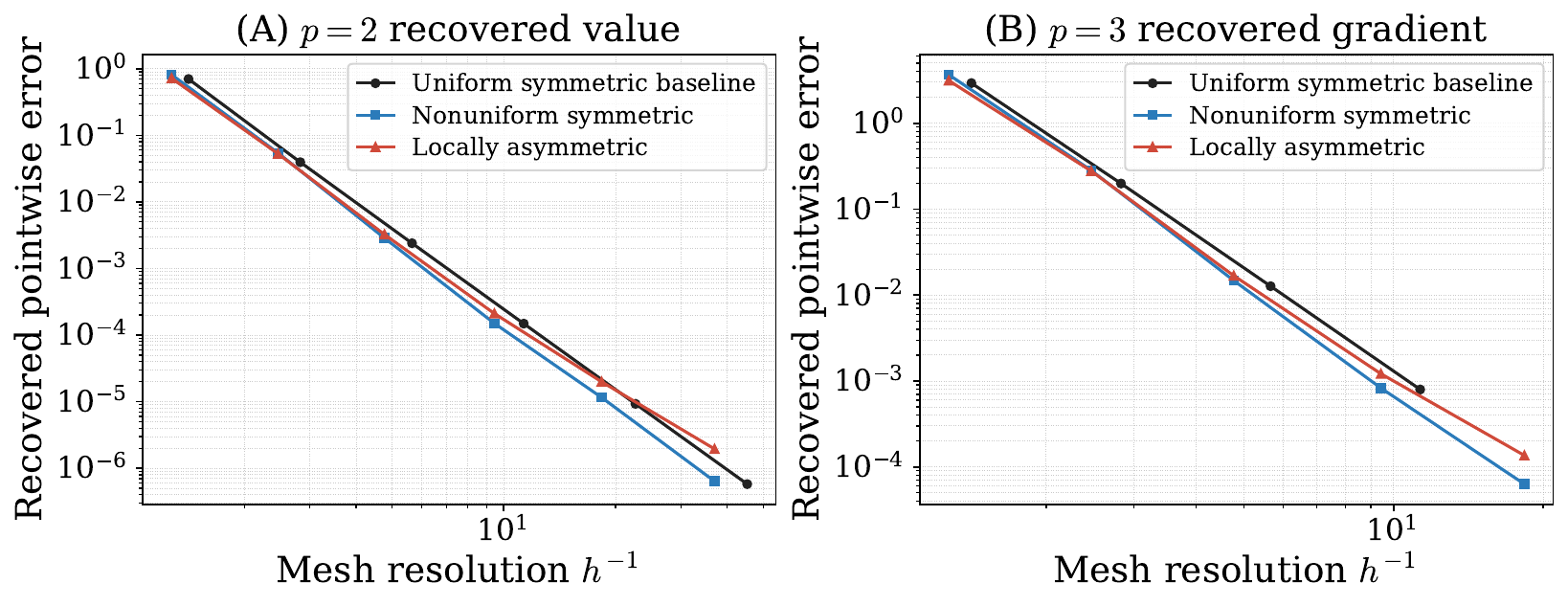}
\caption{Recovered \(p=2\) bending-moment values (A) and \(p=3\) gradients (B) on the three mesh families.}
\label{fig:HHJ-recovery-symmetry-comparison}
\end{figure}

Family~II retains approximately fourth-order recovery for both quantities
predicted to superconverge, whereas Family~III deteriorates on the finer meshes.  At the
finest reported levels, the $p=2$ value rates are $4.136$ and $3.315$
for Families~II and III, while the $p=3$ gradient rates are $3.871$ and
$3.292$, respectively.  Thus nonuniformity alone does not remove the
gain, but breaking the local reflection structure affects it.  Because
Family~III intentionally violates
Assumption~\ref{ass:HHJ-local-reflection-symmetry}, these results
illustrate the symmetry mechanism but do not establish a universal
asymmetric convergence order.

\FloatBarrier
\bibliographystyle{amsplain-nobysame}
\bibliography{ref}
\end{document}